\documentclass[11pt]{article}

\usepackage{color}
\usepackage{algorithm}
\usepackage{algorithmic}
\usepackage{amsmath,url,cases,supertabular}
\usepackage{amssymb,mathtools,multirow}
\usepackage{enumerate,makecell}
\usepackage{amsfonts}
\usepackage{amsthm}
\usepackage{graphicx,epstopdf}
\usepackage{color,verbatim}
\usepackage{mathrsfs,subeqnarray,subfigure}
\usepackage{listings,array}
\usepackage{booktabs}
\usepackage{dashrule}
\usepackage{arydshln}

\usepackage{makecell}

\usepackage[left=1in,top=1in,right=1in,bottom=1in,a4paper]{geometry}

\numberwithin{equation}{section}

\newtheorem{theorem}{Theorem}[section]
\newtheorem{corollary}[theorem]{Corollary}
\newtheorem{lemma}[theorem]{Lemma}
\newtheorem{proposition}[theorem]{Proposition}
\newtheorem{definition}[theorem]{Definition}

\newtheorem{assumption}[theorem]{Assumption}
\newtheorem{remark}[theorem]{Remark}

\newcommand{\etal}{ et al. }

\newcommand{\ba}{\begin{array}}
\newcommand{\ea}{\end{array}}

\newcommand{\bit}{\begin{itemize}}
\newcommand{\eit}{\end{itemize}}
\newcommand{\be}{\begin{equation}}
\newcommand{\ee}{\end{equation}}
\newcommand{\bea}{\begin{eqnarray}}
\newcommand{\eea}{\end{eqnarray}}

\newcommand{\orth}{\mathsf{St}_{d,r}}
\newcommand{\stief}{\mathsf{St}_{d,r}}
\newcommand{\grass}{\mathsf{Gr}_{d,r}}
\newcommand{\mskew}{\mathrm{skew}}
\newcommand{\mF}{f}

\newcommand{\Rbb}{\mathbb{R}}

\newcommand{\st}{\mbox{s.t.}}

\newcommand{\mtr}{\mbox{tr}}

\newcommand{\tF}{\widetilde{F}}

\newcommand{\nn}{\nonumber}
\newcommand{\diag}{\text{diag}}

\newcommand{\proj}{\mathcal{P}}

\newcommand{\LBB}{\mathrm{LBB}}

\newcommand{\Rmn}[1]{\uppercase\expandafter{\romannumeral#1}}

\newcommand{\rblack}[1]{{\color{black}{#1}}}
\newcommand{\Fsf}{\mathsf{F}}
\newcommand{\Tsf}{\mathsf{T}}

\newcommand{\Ebb}{\mathbb{E}}

\newcommand{\Bsf}{\mathsf{B}}

\newcommand{\llangle}{\left \langle}
\newcommand{\rrangle}{\right \rangle}

\newcommand{\Gcal}{\mathcal{G}}
\newcommand{\Dcal}{\Gcal^{\mathrm{R}}}
\newcommand{\Pcal}{\mathcal{P}}
\newcommand{\Fcal}{\mathcal{F}}
\newcommand{\Kcal}{\mathcal{K}}
\newcommand{\Ocal}{\mathcal{O}}
\newcommand{\Tbf}{\mathbf{T}}
\newcommand{\Rscr}{\mathcal{R}}

\numberwithin{equation}{section}
\newcommand{\Pbb}{\mathbb{P}}
\newcommand{\Dbf}{\mathbf{D}}

\newcommand{\rgrad}{\mathrm{grad}}
\newcommand{\up}{\mathrm{up}}

\newcommand{\Mcal}{\mathcal{M}}

\newcommand{\ubf}{\mathbf{u}}
\newcommand{\abf}{\mathbf{a}}
\newcommand{\bbf}{\mathbf{b}}
\newcommand{\cbf}{\mathbf{c}}
\newcommand{\dbf}{\mathbf{d}}
\newcommand{\fbf}{\mathbf{f}}

\newcommand{\1}{\mathbf{1}}
\numberwithin{theorem}{section}

\begin{document}

\title{Vector Transport-Free SVRG with General Retraction for Riemannian  Optimization: Complexity Analysis and Practical Implementation}
\author{Bo Jiang \thanks{School of Mathematical Sciences, Key Laboratory for NSLSCS of Jiangsu Province,
Nanjing Normal University, Nanjing 210023, China. Email: {\tt jiangbo@njnu.edu.cn.}}  \and Shiqian Ma \footnote{Department of Systems Engineering and Engineering Management, The Chinese University of Hong Kong, Shatin, N. T., Hong Kong. Email: {\tt sqma@se.cuhk.edu.hk.}}
\and Anthony Man-Cho So \thanks{Department of Systems Engineering and Engineering Management, and, by courtesy, CUHK–BGI Innovation Institute of Trans–omics, The Chinese University of Hong Kong, Shatin, N. T., Hong Kong. E–mail:
{\tt manchoso@se.cuhk.edu.hk.}}
\and Shuzhong Zhang \thanks{Department of Industrial and Systems Engineering, University of Minnesota, Minneapolis, MN
55455, USA. Email: {\tt zhangs@umn.edu}.}
}
\date{\today}
\maketitle
\begin{abstract}
In this paper, we propose a vector transport-free stochastic variance reduced gradient (SVRG) method with general retraction for empirical risk minimization over Riemannian manifold. Existing SVRG methods on manifold usually consider a specific retraction operation, and involve additional computational costs such as parallel transport or vector transport. The vector transport-free SVRG with general retraction  we propose in this paper handles general retraction operations, and do not need additional computational costs mentioned above. As a result, we name our algorithm S-SVRG, where the first ``S" means simple. We analyze the iteration complexity of S-SVRG for obtaining an $\epsilon$-stationary point and its local linear convergence by assuming the \L ojasiewicz inequality, which naturally holds for PCA and holds with high probability for matrix completion problem. We also incorporate the Barzilai-Borwein step size and design a very practical S-SVRG-BB method. Numerical results on PCA and matrix completion problems are reported to demonstrate the efficiency of our methods.
\end{abstract}

\section{Introduction}
The stochastic variance reduced gradient (SVRG) method proposed by Johnson and Zhang \cite{johnson2013accelerating} has been shown to be very effective for empirical risk minimization problems that involve large-scale training data set in the objective. There have been many variants of SVRG for solving nonsmooth and convex problem \cite{xiao2014proximal}, nonconvex problem \cite{allen2016variance,reddi2016stochastic,aravkin2016smart}, and optimization on manifold \cite{kasai2016riemannian,xu2016stochastic,zhang2016fast}. In this paper, we propose a vector transport-free SVRG with general retraction (S-SVRG) that minimizes empirical risk over manifold:
\be \label{prob:manifold-0}
  \min_{X \in \Rbb^{d \times r}}  f(X) \coloneqq \frac1n \sum_{i = 1}^n f_i(X), \quad \st \quad X\in \Mcal,
\ee
where $\Mcal$ denotes a Riemannian manifold, and $f_i(X) \colon \Mcal \rightarrow \Rbb$ is differentiable.
Although \cite{kasai2016riemannian,xu2016stochastic,zhang2016fast,aravkin2016smart} also studied SVRG for solving \eqref{prob:manifold-0}, we will discuss later that our S-SVRG is more general and efficient, in the sense that it can handle general objective function, general retraction operation, and does not need additional computational costs such as parallel transport and vector transport that are required in \cite{kasai2016riemannian,xu2016stochastic,zhang2016fast}.

For the ease of presentation, we mainly focus on discussing the case when the manifold in \eqref{prob:manifold-0} is the set of orthonormal matrices:
\be \label{prob:manifold}
  \min_{X \in \Rbb^{d \times r}}  f(X) \coloneqq \frac1n \sum_{i = 1}^n f_i(X), \quad \st \quad X\in \orth \coloneqq \{X \in \Rbb^{d \times r} \colon X^{\Tsf} X = I_r\},
\ee
where $f_i(X) \colon \orth \rightarrow \Rbb$ is differentiable, $I_r$ is the $r$-th order identity matrix, and $\orth$ is the compact Stiefel manifold.
The Grassmann manifold $\grass$ takes the quotient representation as $\stief / \mathsf{St}_r$  \cite{edelman1998geometry}, here $\mathsf{St}_r \coloneqq \mathsf{St}_{r,r}$.
Problem \eqref{prob:manifold} has wide applications such as principal component analysis (PCA) \cite{jolliffe2014principal}, the Karcher mean problem \cite{kasai2016riemannian}, the joint diagonalization problem \cite{theis2009soft}, the domain adaptation problem \cite{politz2016interpretable}, and a recovering problem in dictionary learning \cite{sun2016complete}, just to name a few.

\subsection{Related works}

There are mainly two classes of SVRG methods for manifold optimization. \cite{kasai2016riemannian,xu2016stochastic,zhang2016fast} belong to the first class and they construct the stochastic variance reduced Riemannian gradients by invoking parallel transport or vector transport. Specifically, Kasai, Sato and Mishra \cite{kasai2016riemannian} mainly consider SVRG on Grassmann manifold, and the retraction considered is exponential map \footnote{For general Riemannian manifold, it is stated in \cite{absil2008optimization} that ``computing the exponential is, in general, a computationally daunting task''. For a comparison of the computational cost for the exponential map in $\orth$ and other retractions, see \cite{jiang2015framework, gao2016new}.}. Moreover, the algorithm in \cite{kasai2016riemannian} requires diminishing step size to guarantee the global convergence and its local linear convergence requires the positiveness of the Riemannian Hessian at a non-degenerate local minimizer. Xu and Ke \cite{xu2016stochastic} consider SVRG for eigenvalue problem whose objective function is quadratic and the retraction used is polar decomposition. The linear convergence of the algorithm in \cite{xu2016stochastic} requires that the initial point is sufficiently close to the optimal solution. Zhang, Reddi and Sra \cite{zhang2016fast} consider SVRG for general Riemannian manifold (R-SVRG), and the retraction used is the exponential map.
However, R-SVRG restricts the objective function $f$ on a compact set $\mathcal{X}$ of the considered manifold.
\cite{zhang2016fast} proves that the IFO-calls complexity (i.e., the total number of component gradient evaluations) of R-SVRG for achieving an $\epsilon$-stationary point is $O(\zeta^{\frac12} {n^{2/3}}/{\epsilon} + n)$, where $\zeta$ depends on the sectional curvature and the diameter of $\mathcal{X}$, and this result also requires that $f$ is geodesically $L$-smooth and the sectional curvature of $\mathcal{X}$ has finite lower and upper bounds. Linear convergence of R-SVRG is shown under the condition that $f$ is globally gradient dominated.
Besides, \cite{zhang2016fast} also shows the linear convergence of R-SVRG under the assumption that $f$ is geodesically convex. However, it should be noted that every geodesically convex function on a compact manifold is a constant \cite{bishop1969manifolds}.
Papers in the second class include \cite{shamir2015stochastic,shamir2015fast,weijie2016quadratic,aravkin2016smart} and they do not need parallel transport or vector transport. Shamir proposes the VR-PCA algorithm \cite{shamir2015stochastic,shamir2015fast} for solving PCA problem whose objective function is quadratic. Linear convergence is established under the assumption that the initial point is sufficiently close to the optimal solution. The Stiefel-SVRG proposed by Wu \cite{weijie2016quadratic} is also for solving PCA problem, and linear convergence to stationary point is shown by proving that the \L ojasiewicz inequality holds for PCA problem.
Aravkin and Davis \cite{aravkin2016smart} propose a very general SVRG algorithm for solving nonconvex and nonsmooth composite problems that include \eqref{prob:manifold} as a special case. The IFO-calls complexity of algorithm in \cite{aravkin2016smart} is also shown to be $O(n^{2/3}/\epsilon+n)$. However, the retraction operation considered in \cite{aravkin2016smart} is gradient projection, and it is not clear how to extend the result to other retractions. The differences of the existing methods and our S-SVRG are summarized in Table \ref{table:result}.

\begin{table}[!htbp]
\centering
\caption{Comparison of SVRG for manifold optimization: ``IFO'' denotes whether the IFO-calls complexity is analyzed} \label{table:result}
~\\
\renewcommand{\arraystretch}{1.7}
\begin{footnotesize}
\begin{tabular}{@{\hspace{0.2mm}}l@{\hspace{3.2mm}}l@{\hspace{3.2mm}}l@{\hspace{3.2mm}}l@{\hspace{3.2mm}}l@{\hspace{3.2mm}}c@{\hspace{0.2mm}}}
\Xhline{0.6pt}
method & obj. & manifold & retraction & additional cost &  IFO  \\
\Xhline{1.1pt}
VR-PCA \cite{shamir2015fast}& quadratic & Stiefel & gradient projection & --- &   $\times$ \\
Stiefel-SVRG \cite{weijie2016quadratic} & quadratic & Stiefel & general & --- &    $\times$ \\
R-SVRG \cite{kasai2016riemannian} & general & Grassmann & exponential &  parallel translation  & $\times$ \\
SVRRG \cite{xu2016stochastic} & quadratic & Stiefel & polar decomposition & vector transport   & $\times$ \\
RSVRG  \cite{zhang2016fast}  & general & compact& exponential & parallel translation &    $\surd$ \\
SMART-SVRG \cite{aravkin2016smart}  &   general & general & gradient projection & --- &  $\surd$ \\
Our S-SVRG & general & general & general & ---  & $\surd$ \\
\Xhline{0.6pt}
\end{tabular}
\end{footnotesize}
\end{table}

\subsection{Our contributions}

Our main contributions lie in several folds.
(i) Our S-SVRG handles general objective function, general manifold, and general retraction steps (including all the seven retractions that will be discussed in Appendix \ref{section:retractions}), and it does not need additional costs such as parallel transport or vector transport. (ii) We analyze the convergence and IFO-calls complexity of S-SVRG, and we do not need the various conditions in \cite{zhang2016fast}. As a by-product, we also analyze the convergence and IFO-calls complexity of a vector transport-free randomized stochastic gradient descent with general retraction (S-SGD) for solving \eqref{prob:manifold}. (iii) We prove the local linear convergence of S-SVRG using \L ojasiewicz inequality, which holds naturally for PCA problem and holds with high probability for matrix completion problems. Therefore, the condition required for local linear convergence is weaker than the ones used in the previous works \cite{kasai2016riemannian,xu2016stochastic,zhang2016fast,aravkin2016smart}. (iv) We incorporate the Barzilai-Borwein step size to S-SVRG and design a very practical S-SVRG-BB method, and this resolves the issue of choosing step size in S-SVRG.

\subsection{Organization}
The rest of this paper is organized as follows. We introduce some preliminaries on manifold optimization in Section \ref{section:preliminaries}. We propose our S-SVRG method for solving \eqref{prob:manifold} and analyze its IFO-calls complexity and local linear convergence in Section  \ref{section:usvrg}. As a by-product, we also provide an analysis of SGD for solving \eqref{prob:manifold}. Extensions to more general manifolds are studied in Section \ref{section:extensions}. We propose the S-SVRG-BB method in Section \ref{section:usvrgBB}. Numerical results on PCA and matrix completion problems are presented in Section \ref{section:numerical}. Finally, we conclude the paper in Section \ref{section:conclusion}.

\section{Preliminaries}\label{section:preliminaries}

Throughout this paper, we make the following assumption for \eqref{prob:manifold}.
\begin{assumption}\label{assumption_lipschitz}
 $f_i(X)\colon \orth \rightarrow \Rbb$ is differentiable, and its Euclidean gradient $\nabla f_i(X)$ is $L$-Lipschitz continuous over $\orth$, {\it i.e.},
\be\label{equ:g:lips}
\|\nabla f_i(X) - \nabla f_i(Y)\|_{\Fsf} \leq L \|X - Y\|_{\Fsf}, \quad \forall  X, Y \in \orth.
\ee
\end{assumption}
It follows that $\nabla f(X)$ is also $L$-Lipschitz continuous over $\orth$.

We first introduce some basic notions of optimization on Riemannian manifold $\Mcal$.  For each $X \in \Mcal$, the tangent space is denoted by
$\Tbf_X \Mcal$.   We  define the \emph{inner product} on $\Tbf_X \Mcal$ as  $\langle \cdot, \cdot \rangle_X$ and the corresponding induced norm  $\|E\|_X = \sqrt{\langle E, E\rangle_X}$ is an equivalent norm to the Frobenius norm, namely, there exist $\nu, \gamma > 0$ such that
\be  \label{equ:gradf:norm}
\nu \|E\|_{\Fsf}^2 \leq \langle E, E \rangle_{X} \leq \gamma  \|E\|_{\Fsf}^2, \quad \forall E \in \Tbf_X \Mcal.
\ee

The Riemannian gradient  $\rgrad\, f(X)$  is the unique element of $\Tbf_X \Mcal$  satisfying
\be\label{equ:gradf}
\langle \rgrad\, f(X), E \rangle_{X}  = \mathrm{D}f(X)[E] = \langle \nabla f(X), E \rangle,  \quad \forall E \in \Tbf_X \Mcal,
\ee
where $\langle \cdot, \cdot \rangle$ is the Euclidean inner product.

The feasible gradient descent method is based on the notion of retraction. Given any $X \in \Mcal$,  the {\it retraction}   $\Rscr(t) \coloneqq \Rscr(X, t E)$ along the direction $E \in \Tbf_X \Mcal$ is a smooth map from $\Tbf_X \Mcal$ to $\Mcal$ that satisfies
 \be \label{equ:Yt}
\quad \Rscr(0)  = X, \ \Rscr'(0)  = E, \quad \forall t \in [0, T_{\Rscr}],
 \ee
 where $\Rscr'(0) = \frac{\mathrm d}{\mathrm d t} \Rscr(t)\big |_{t = 0}$.
Starting from a given initial point $X^0$, the feasible method based on the retraction updates the iterates as
\be\label{alg:fg}
X^{k+1} = \Rscr\big({X^k}, \rblack{\tau_k} E^k\big),\   k = 0, 1, 2, \ldots
\ee
where $\tau_k > 0$ is the step size and $E^k \in \Tbf_X \Mcal$. 


In the following we give the definitions of stationary point and $\epsilon$-stationary point of \eqref{prob:manifold-0}.

\begin{definition}[Stationary point]\label{lemma:optimality}
 $X \in \Mcal$ is called  a stationary point of problem \eqref{prob:manifold-0} if $\rgrad\, f(X) =0. $
\end{definition}


\begin{definition}[$\epsilon$-stationary point]
$X \in \Mcal$ is called  an $\epsilon$-stationary point of problem \eqref{prob:manifold-0} if $\|\rgrad\, f(X)\|_{\Fsf}^2 \leq \epsilon$.
\end{definition}

\begin{definition}[stochastic $\epsilon$-stationary point]\label{def:stoc-eps-stat-point}
Suppose  $X^r \in \Mcal$ is returned by a stochastic feasible method for \eqref{prob:manifold-0}, we call $X^r$  a stochastic $\epsilon$-stationary point if
$$\Ebb \big[\|\rgrad\, f(X^r) \|_{\Fsf}^2 \big] \leq \epsilon,$$
where the expectation is taken with respect to the whole stochasticity of the algorithm.
\end{definition}

With slight abuse of notation, we still use $\rgrad\,f(X)$ to denote Riemannian gradient on $\orth$. We now give a complete description of $\rgrad\,f(X)$.
For each $X \in \stief$, the tangent space at $X$ is referred to $\Tbf_X \stief :=\{Z \in \Rbb^{d \times r} \colon X^{\Tsf}Z+Z^{\Tsf}X=0\}$. For any $\rho > 0$, define the inner product on $\Tbf_X \stief$ as
$\langle E_1, E_2 \rangle_X =  \llangle E_1, \mathbf{P}_{\rho,X} E_2\rrangle, \forall E_1, E_2 \in \Tbf_X \stief$, where  $\mathbf{P}_{\rho,X} =  I_d - \big(1 - {1}/{(4 \rho)}\big) XX^{\Tsf}$.
Each point in $\grass$ is essential an equivalent class $[X]  = \{XQ_r \colon  Q_r \in \mathsf{St}_r\}$, where $\mathsf{St}_r$ stands for $\Ocal_{r,r}$ for short. The tangent space at $[X]$ is given as
$\Tbf_{[X]} \grass = \{Z \in \Rbb^{d \times r} \colon X^{\Tsf} Z = 0\}$ and the corresponding inner product is always taken as $\langle E_1, E_2 \rangle_{[X]} = \langle E_1, E_2 \rangle, \ \forall E_1,E_2 \in \Tbf_{[X]} \grass$; see \cite{edelman1998geometry}. For simplicity of notation, we  denote $[X]$ by $X$. Given $\rho \geq 0$, define the operator
\be\label{equ:Drho:X:Y}
\Dbf_{\rho}(X,Y) \coloneqq (I_r - XX^{\Tsf}) Y + 4\rho X \mskew(X^{\Tsf} Y),
\ee
where $\mskew(X^{\Tsf} Y) = (X^{\Tsf} Y - Y^{\Tsf} X)/2$. From \eqref{equ:gradf}, we know the Riemannian gradient on $\orth$ can be defined as
$$
\rgrad\,f(X) = \Dbf_{\rho}(X,\nabla f(X))
$$
with $\rho = 0$ for $\grass$ and $\rho > 0$ for  $\stief$. Note that when $X^{\Tsf} \nabla f(X) = \nabla f(X)^{\Tsf} X$, we have $\rgrad\, f(X) \equiv \Dbf_{0}(X,\nabla f(X))$. Besides, \eqref{equ:gradf:norm} holds for $\orth$ with $\nu = \min\left\{1, 1/(4\rho)\right\}$ and $\gamma = 1$.

\section{A S-SVRG method for problem \eqref{prob:manifold}}\label{section:usvrg}

In this section we propose a S-SVRG method for solving \eqref{prob:manifold}. 
We first establish a sufficient decrease property of retraction of $\orth$, which plays an important role in establishing the complexity results of S-SVRG.

\subsection{Sufficient decrease property} \label{subsection:lemmas}
Consider a retraction  $\Rscr(t)$ of $\orth$ with $\Rscr(0) = X$. \footnote{For the retractions considered in this paper, we know that $T_{\Rscr} = + \infty$; see Appendices \ref{section:retraction:stiefel} and \ref{section:retraction:grass}.} We aim to establish the so-called sufficient decrease property of $f(\Rscr(t))$, i.e., there is a sufficient reduction from $f(X)$ to $f(\Rscr(t))$. By the compactness of $\orth$, there exists a constant $C > 0$ such that
$$\|  \nabla f(X) \|_{\Fsf} \leq C, \quad \forall X \in \orth. $$
Following the proof of Lemma 3 in \cite{boumal2016global}, we know that $\Rscr(t)$ satisfies the following properties. 
\begin{proposition}\label{proposition:retraction}
There exist constants $L_1 \geq 1$, $L_2 > 0$ such that
\begin{align}
\| \rblack{\Rscr(t)} - \rblack{\Rscr(0)}\|_{\Fsf} \leq L_1 t \|\rblack{\Rscr'(0)}\|_{\Fsf}, \label{equ:retraction:b1}\\
\|\rblack{\Rscr(t)} - \rblack{\Rscr(0)} - t \rblack{\Rscr'(0)}\|_{\Fsf} \leq L_2 t^2 \|\rblack{\Rscr'(0)}\|_{\Fsf}^2, \label{equ:retraction:b2}
\end{align}
hold for any $t \geq 0$. 
\end{proposition}

We have the following remarks regarding Proposition \ref{proposition:retraction}. 
For the particular case $\Mcal = \Rbb^{n\times r}$ and $\rblack{\Rscr(t)} = X + tE$, we always have $L_1 = 1$ and $L_2 = 0$, for given $X,E\in\Rbb^{n\times r}$. For retractions on $\orth$, we can compute $L_1$ and $L_2$ explicitly. For example, when $\rblack{\Rscr(t)}$ is polar decomposition, we have $L_1=1$, $L_2=1/2$, and when $\rblack{\Rscr(t)}$ is QR factorization, we have $L_1=1 + \sqrt{2}/{2}$, $L_2=\sqrt{10}/{2}$. Note that these estimations of $L_1$ and $L_2$ are much better than those in \cite{dai2016conjugate, weijie2016quadratic}. The corresponding proofs are given in Appendix \ref{section:L1:L2}.

Second, for any $Z \in \Rbb^{d \times r}$, there holds that
 \begin{align}
  \llangle Z, \rblack{\Rscr(t)} - \rblack{\Rscr(0)} \rrangle ={} &  t \llangle Z, \rblack{\Rscr'(0)} \rrangle + \llangle Z, \rblack{\Rscr(t)} - \rblack{\Rscr(0)} - t\rblack{\Rscr'(0)} \rrangle \nn \\
  \leq{}& t \llangle Z, \rblack{\Rscr'(0)} \rrangle +  \|Z\|_{\Fsf} \|\rblack{\Rscr(t)} - \rblack{\Rscr(0)} - t\rblack{\Rscr'(0)} \|_{\Fsf}  \nn \\
  \leq{}& t \llangle Z, \rblack{\Rscr'(0)} \rrangle + L_2 t^2 \|Z\|_{\Fsf} \|\rblack{\Rscr'(0)}\|_{\Fsf}^2,\label{equ:ZYtY0:00}
 \end{align}
where the second inequality is due to \eqref{equ:retraction:b2}. Inequality \eqref{equ:ZYtY0:00} will be used later in our analysis.

We are now ready to present the sufficient decrease property, whose proof can be found in Lemma 3 of \cite{boumal2016global}. For completeness, we give a simple proof here. 
\begin{lemma}\label{lemma:pullback}
For any $t \geq 0,$ there holds that
\begin{align} \label{equ:sufficientDecrease:fYt}
f(\rblack{\Rscr(t)})\leq f(X)  + t \langle \rgrad\, f(X), \rblack{\Rscr'(0)} \rangle_X + \frac{\hat L}{2} t^2 \|\rblack{\Rscr'(0)}\|_{\Fsf}^2,
\end{align}
where $\hat L = 2 L_2 C +  L_1^2 L.$
\end{lemma}

\begin{proof}
From \eqref{equ:g:lips}, we know that $\nabla f(X)$ is $L$-Lipschitz continuous. Using $\rblack{\Rscr(0)} = X$, we have
\be \label{equ:lemma:pullback:a0}
f(\rblack{\Rscr(t)}) \leq f(X) + \langle \nabla f(X), \rblack{\Rscr(t)} - \rblack{\Rscr(0)}  \rangle  + \frac{L}{2} \|\rblack{\Rscr(t)}  - \rblack{\Rscr(0)} \|_{\Fsf}^2.
\ee
It follows from \eqref{equ:gradf} that $\langle \nabla f(X), \rblack{\Rscr'(0)} \rangle = \langle \rgrad\, f(X), \rblack{\Rscr'(0)} \rangle_X$, which together with \eqref{equ:ZYtY0:00} and \eqref{equ:retraction:b1} implies \eqref{equ:sufficientDecrease:fYt}.
\end{proof}

It is shown in Lemma 3 of \cite{boumal2016global} that if $\Mcal$ is a compact Riemannian submanifold of some Euclidean space and $f$ has Lipschitz continuous gradient in the convex hull of $\Mcal$, then \eqref{equ:sufficientDecrease:fYt} must hold because  \eqref{equ:retraction:b1} and \eqref{equ:retraction:b2} hold. However, if  $\Mcal$ is not compact, it remains unknown whether \eqref{equ:sufficientDecrease:fYt} holds for some universal $T_{\Rscr}$ and $t \in [0, T_{\Rscr}]$.

\subsection{A S-SVRG method} \label{subsection:usvrg}
Our S-SVRG method is described in Algorithm \ref{alg:usvrg}. The random event in Line 6 of Algorithm \ref{alg:usvrg} is denoted by $\xi_{s,k}$. Clearly, $\xi_{s,k}$ is mutually independent of each other. \rblack{For fixed $s$, we simply denote   $X^{s,k}$,  $\xi_{s,k}$, $\tau_{s}$ and $\Dcal(X^{s,k}, \xi_{s,k})$ respectively by $X^k$, $\xi_k$, $\tau$ and $\Dcal(X^k, \xi_k)$.}  With the stochastic {\it Euclidean} gradient $\Gcal(X^{k}, \xi_{k}) $ in hand, we can compute the stochastic {\it Riemannian} gradient $\Dcal(X^k, \xi_k)$ such that
\be \label{equ:Dcalrho:relation}
\langle \Dcal(X^k, \xi_k), E \rangle_{X^k}  =  \langle \Gcal(X^k, \xi_k), E \rangle, \quad \forall  E \in \Tbf_{X^k} \orth,
\ee
which gives $\Dcal(X^{k}, \xi_{k}) = \Dbf_{\rho}\left(X^{k}, \Gcal(X^{k}, \xi_{k})\right)$, where the operator $\Dbf_{\rho}(\cdot, \cdot)$ is defined in \eqref{equ:Drho:X:Y}.

\begin{algorithm}[!hbtp]
     \caption{A S-SVRG method for problem \eqref{prob:manifold}}\label{alg:usvrg}
     \begin{algorithmic}[1]
     \STATE{Given $X^{0,0}\in \orth$, \rblack{the retraction type $\Rscr(\cdot,\cdot)$,} 
     the direction parameter $\rho \geq 0$ in $\Dbf_{\rho}$.}
\STATE{Choose the maximal inner iteration number $K \geq 1$ and the mini-batch size $|\Bsf| \geq 1$.}
 \FOR{$s = 0, \ldots, S-1$}
\STATE Compute the full {\it Euclidean} gradient $\nabla f(X^{s,0})$ and set the \rblack{step size} $\tau_s > 0$. 
\FOR{$k = 0, \ldots, K-1$}
\STATE Generate a uniformly random sample $\Bsf \subseteq  \{1, \ldots, n\}$ with replacement.
\STATE Compute the stochastic {\it Euclidean} gradient $\Gcal(X^{s,k}, \xi_{s,k})$ as 
\be \label{equ:sum:procedure:SFO}
\Gcal(X^{s,k}, \xi_{s,k}) = \nabla f(X^{s,0}) + \frac{1}{|\Bsf|} \sum_{i \in \Bsf} \big( \nabla f_i(X^{s,k}) -  \nabla f_i(X^{s,0})\big).
\ee
\STATE Compute the stochastic {\it Riemannian} gradient  as \footnotemark
\be \label{equ:sRG}
\Dcal(X^{s,k}, \xi_{s,k}) = \Dbf_{\rho}\big(X^{s,k}, \Gcal(X^{s,k}, \xi_{s,k})\big).
\ee
 \STATE Update $X^{s,k+1}$ \rblack{along the direction $-\Dcal(X^{s,k}, \xi_{s,k})$, i.e.,}
 \be \label{equ:Xk:update:usvrg}
X^{s, k+1} = \Rscr\big({X^{s,k}}, -\tau_{s} \Dcal(X^{s,k}, \xi_{s,k})\big).
\ee
 \ENDFOR
\STATE{Set  $X^{s+1,0} \coloneqq X^{s,K}$. Set $X_r^s$ to be  $X^{s,k}$ with probability $p_{s,k}, k = 0, \ldots, K$.}
\ENDFOR
\RETURN{$X_r$ uniformly from $\{X_r^{s}\}$ with $s \in \{0,1, \ldots, S-1\}$.}
   \end{algorithmic}
  \end{algorithm}
\footnotetext{For some special retractions, such as the gradient projection and gradient reflection retractions on $\orth$, it is not necessary to compute $\Dcal(X^{s,k}, \xi_{s,k})$ explicitly.}

We now show that the stochastic Riemannian gradient $\Dcal(X^k, \xi_k)$ is unbiased and its variance can be well controlled.

\begin{lemma}\label{lemma:usvrg:unbiased:bound_varaince}
For the sequences generated by  Algorithm \ref{alg:usvrg}, it holds that
\begin{align}
&\Ebb_{\xi_{k}} \big[ \Dcal(X^k, \xi_k) - \rgrad\, f(X^k) \big] =   0,  \label{equ:svrg:Dcal:unbiased} \\
&\Ebb_{\xi_k} \big[ \| \Dcal(X^k, \xi_k) - \rgrad\, f(X^k) \|_{\Fsf}^2 \big] \leq \frac{L^2}{\nu^2 |\Bsf|} \|X^k - X^0\|_{\Fsf}^2, \label{equ:svrg:Dcal:variance}
\end{align}
\rblack{where the constant $\nu$, defined in \eqref{equ:gradf:norm}, equals to $\min\left\{1, 1/(4\rho)\right\}$.}
\end{lemma}

\begin{proof}
It is easy to see that $\Ebb_{\xi_{k}} \big[ \Gcal(X^k, \xi_k) - \nabla f(X^k) \big] = 0$. Taking the expectation over $\xi_k$ on both sides of \eqref{equ:Dcalrho:relation}, we have
\be \label{equ:Dcalrho:relation:expectation}
\langle \Ebb_{\xi_k} [\Dcal(X^k, \xi_k)], E \rangle_{X^k} 
= \langle \nabla f(X^k), E \rangle, \quad \forall E  \in \Tbf_{X^k} \orth,
\ee
which together with \eqref{equ:gradf} and the uniqueness of $\rgrad\, f(X^k)$ implies  \eqref{equ:svrg:Dcal:unbiased}.

We now prove \eqref{equ:svrg:Dcal:variance}.
By \eqref{equ:Dcalrho:relation} and \eqref{equ:gradf}, we have 
\begin{align} \label{equ:proof:lemma:variance:cc0}
\langle \Dcal(X^k, \xi_k) - \rgrad\, f(X^k), E \rangle_{X^k}  =  \langle \Gcal(X^k, \xi_k) - \nabla f(X^k), E \rangle, \quad \forall E \in \Tbf_{X^k} \orth.
\end{align}
Letting $E = \Dcal(X^k, \xi_k) - \rgrad\, f(X^k)$ in \eqref{equ:proof:lemma:variance:cc0}, we have from \eqref{equ:gradf:norm} that
\begin{align}
\nu \|\Dcal(X^k, \xi_k) - \rgrad\, f(X^k)\|_{\Fsf}^2 \leq \|\Gcal(X^k, \xi_k) - \nabla f(X^k)\|_{\Fsf} \cdot \|\Dcal(X^k, \xi_k) - \rgrad\, f(X^k)\|_{\Fsf}, \nn
\end{align}
which yields
\begin{align}\label{equ:proof:lemma:variance:d0}
\|\Dcal(X^k, \xi_k) - \rgrad\, f(X^k)\|_{\Fsf}^2 \leq \frac{1}{\nu^2} \|\Gcal(X^k, \xi_k) - \nabla f(X^k)\|_{\Fsf}^2.
\end{align}
Taking the expectation over $\xi_k$ on both sides of \eqref{equ:proof:lemma:variance:d0} leads to
\be
\Ebb_{\xi_k} \big[\|\Dcal(X^k, \xi_k) - \rgrad\, f(X^k)\|_{\Fsf}^2\big] \leq \frac{1}{\nu^2} \cdot \Ebb_{\xi_k} [\|\Gcal(X^k, \xi_k) - \nabla f(X^k)\|_{\Fsf}^2].\label{equ:D:variance}
\ee
 On the other hand, we have
\be
\Ebb_{\xi_k} \big[ \| \Gcal(X^k, \xi_k) - \nabla f(X^k) \|_{\Fsf}^2 \big] \leq  \frac{1}{n |\Bsf|} \frac{n - |\Bsf|}{n - 1}  \sum_{i = 1}^n \|\nabla f_i(X^k) - \nabla f_i(X^0)\|_{\Fsf}^2  \leq  \frac{L^2}{|\Bsf|}  \|X^k - X^0\|_{\Fsf}^2, \nn
\ee
where the first inequality comes from Lemma 4 in \cite{konevcny2016mini}, and the second one is due to \eqref{equ:g:lips}.
Combining  the above inequality  and \eqref{equ:D:variance},  we have \eqref{equ:svrg:Dcal:variance}.
\end{proof}

The following lemma plays a key role in establishing the iteration complexity. Its proof is relegated to Appendix \ref{subsection:proof:Lemma:recursion}.
\begin{lemma}\label{lemma:recursion}
Consider the sequences $\{\abf_k \geq 0\}, \{\bbf_k \geq 0\}, \{\fbf_k\}$ with $k = 0, \ldots, K$ and $\bbf_0 = 0$. If there exist positive constants $\abf, \bbf \neq 1,\cbf,\dbf$  such that
\begin{align}
\fbf_{k+1}  \leq{} & \fbf_k - \cbf \abf_k + \dbf \bbf_k, \label{equ:lemma:recursion:f:a:b} \\[2pt]
\bbf_{k+1} \leq{} &  \bbf \bbf_k + \abf \abf_k  \label{equ:lemma:recursion:b:a}
\end{align}
hold for $k = 0, \ldots, K- 1$.
Then  we have
\be\label{equ:lemma:recursion:bk:fK}
\fbf_K  \leq \fbf_0 -\sum_{k= 0}^{K-1} \Delta_k \abf_k \quad \mbox{with} \quad \Delta_k =  \cbf - \abf \dbf \rblack{\Gamma (b, K-k)},
\ee
\rblack{where the function $\Gamma(\cdot, \cdot)$ is defined as $\Gamma(z, i) = \frac{(1 + z)^{i-1} - 1}{z}$.}
\end{lemma}

\subsection{Iteration complexity of S-SVRG} \label{subsection:complexity}
We first show that the function value over one epoch, i.e., one outer loop,  has sufficient reduction in expectation.
\begin{lemma}\label{lemma:sufficient:decrease:epoch}
Consider Algorithm \ref{alg:usvrg}. For fixed $s$, it holds that
\begin{align} \label{equ:f:x:K:expectation}
\Ebb_{\xi_{[K-1]}}[\mF(X^{K})]  \leq f(X^0) -  \sum_{k=0}^{K-1}\Delta_k \Ebb_{\xi_{[K-1]}}\big[\| \rgrad\,f(X^k) \|_{\Fsf}^2 \big],
\end{align}
where $\xi_{[K-1]} = (\xi_0, \ldots, \xi_{K-1})$ and
\be \label{equ:Delta:k}
\frac{\Delta_k}{\tau} = \nu - \frac{\hat L\tau}{2}  \left[1 + \left(1 + \frac{2}{\tilde{L}\beta\tau}\right)  \frac{\tilde{L}L^2 \tau^2}{\nu^2 |\Bsf|}  \Gamma_k\right]\!,
\ee
where $\beta > 0$ is a constant, $\tilde{L} = L_1^2 + 4L_2 \sqrt{r}$ and
\be \label{equ:gamma:k}
\Gamma_k = \Gamma\left(2 \beta \tau + \frac{\tilde{L} L^2 \tau^2}{\nu^2 |\Bsf|}, K- k\right).
\ee
\end{lemma}

\begin{proof}
The proof consists of three steps.

First, we establish the sufficient reduction of the function value per inner iteration. By \eqref{equ:svrg:Dcal:unbiased}, we have
\be \label{equ:variance:law}
  \Ebb_{\xi_k} \big[\|\Dcal(X^k, \xi_k)\|_{\Fsf}^2\big] = \Ebb_{\xi_k}\big[\|\Dcal(X^k, \xi_k) - \rgrad\, f(X^k) \|_{\Fsf}^2 \big]  + \|\rgrad\, f(X^k)\|_{\Fsf}^2.
\ee
Letting $\rblack{\Rscr'(0)} = - \Dcal(X^k, \xi_k)$ in \eqref{equ:sufficientDecrease:fYt}, and taking the expectation over $\xi_k$ on both sides of the resulting inequality, we have from \rblack{$X^k = \Rscr(X^k, 0)$, \eqref{equ:Xk:update:usvrg} and} \eqref{equ:variance:law} that
\begin{align} \label{proof:theorem:suffcientDecrease:a1}
  \Ebb_{\xi_k} [\mF(X^{k+1})]
  \leq {}&\mF(X^k)  - \tau \big \langle  \rgrad\, f(X^k), \rgrad\, f(X^k) \big \rangle_{X^k}  + \frac{1}{2} \hat L \tau^2\|\rgrad\, f(X^k)\|_{\Fsf}^2 \nn \\
  & +\frac{1}{2} \hat L \tau^2  \Ebb_{\xi_k} \big[ \|\Dcal(X^k, \xi_k) - \rgrad\, f(X^k)\|_{\Fsf}^2 \big].
\end{align}
Plugging \eqref{equ:svrg:Dcal:variance} into \eqref{proof:theorem:suffcientDecrease:a1} and using \eqref{equ:gradf:norm} with $E = \rgrad\, f(X^k)$, we obtain
\begin{align}
  \Ebb_{\xi_k}[\mF(X^{k+1})]  \leq{}&\mF(X^k) - \frac{\tau}{2}  \left(2 \nu - \hat L  \tau \right) \| \rgrad\, f(X^k) \|_{\Fsf}^2   + \frac{ \hat L L^2\tau^2}{2\nu^2 |\Bsf|}  \|X^k -  X^0\|_{\Fsf}^2. \label{lemma:fX:descent:rsvrg:a1}
\end{align}

Second, we prove the following inequality: 
\begin{align}
{}&\Ebb_{\xi_k} \big[\|X^{k+1} - X^0\|_{\Fsf}^2\big] \nn \\
\leq{} &  \left(1 + 2 \beta \tau + \frac{\tilde{L} L^2  }{\nu^2 |\Bsf|} \tau^2  \right) \|X^k - X^0\|_{\Fsf}^2  +  \tau \left(\tilde{L}\tau + \frac{2}{\beta}\right)  \|\rgrad\, f(X^k)\|_{\Fsf}^2. \label{equ:theorem:xk:recursion}
\end{align}
To prove this, we first obtain from $X^k = \Rscr(X^k, 0)$ and \eqref{equ:Xk:update:usvrg} that
\begin{align}
\|X^{k+1} - X^0\|_{\Fsf}^2 
={} & \|X^k - X^0\|_{\Fsf}^2 + \|\Rscr\big(X^k, -\tau \Dcal(X^k, \xi_k)\big) - \Rscr(X^k, 0)\|_{\Fsf}^2  \nn \\
& +  2\big\langle X^k  - X^0, \Rscr(X^k, -\tau \Dcal(X^k, \xi_k)) - \Rscr(X^k, 0)  \big \rangle.   \label{equ:theorem:xk:recursion:00000}
\end{align} 
It follows from \eqref{equ:ZYtY0:00} \rblack{and $\Rscr'(0) = - \Dcal(X^k, \xi_k)$} that
\begin{align}
& \big\langle X^k  - X^0, \Rscr(X^k, -\tau \Dcal(X^k, \xi_k)) - \Rscr(X^k, 0)  \big\rangle \nn \\ 
\leq{} &\rblack{-\tau} \big\langle X^k  - X^0, \Dcal(X^k, \xi_k) \big\rangle  + L_2  \tau^2 \|X^k - X^0\|_{\Fsf} \|\Dcal(X^k, \xi_k)\|_{\Fsf}^2,\nn
\end{align}
which, together with the fact that $\|X^k - X^0\|_{\Fsf} \leq 2\sqrt{r}$, \eqref{equ:theorem:xk:recursion:00000} and the definition of $\tilde{L}$ implies
\begin{align}
      \|X^{k+1} - X^0\|_{\Fsf}^2 \leq \|X^k - X^0\|_{\Fsf}^2  +  \tilde{L} \tau^2  \|\Dcal(X^k, \xi_k)\|_{\Fsf}^2  \ \rblack{-}\   2 \tau \big\langle X^k  - X^0, \Dcal(X^k,\xi_k) \big\rangle. \label{equ:lemma:Xk:update:a1}
\end{align}
Taking expectation over both sides of \eqref{equ:lemma:Xk:update:a1} with respect to $\xi_k$, and using \eqref{equ:svrg:Dcal:unbiased}, we have
\begin{align}\label{equ:lemma:xk:recursion:cc4}
 \Ebb_{\xi_k}  \big[\|X^{k+1} - X^0\|_{\Fsf}^2\big] \leq \|X^k - X^0\|_{\Fsf}^2 +  \tilde{L} \tau^2   \Ebb_{\xi_k} \big[\|\Dcal(X^k, \xi_k)\|_{\Fsf}^2\big] \ \rblack{-}\   2 \tau \big\langle X^k  - X^0, \rgrad\,f(X^k) \big\rangle.
\end{align}
Combining  \eqref{equ:variance:law} and \eqref{equ:svrg:Dcal:variance}, we have
\be\label{equ:lemma:xk:recursion:cc5}
\Ebb_{\xi_k} \big[\|\Dcal(X^k, \xi_k)\|_{\Fsf}^2\big] \leq \frac{L^2}{\nu^2 |\Bsf|} \|X^k - X^0\|_{\Fsf}^2  + \|\rgrad\,f(X^k)\|_{\Fsf}^2.
\ee
On the other hand, it follows from the  Cauchy-Schwarz inequality that
\be\label{equ:lemma:xk:recursion:cc6}
\rblack{-}\ \big\langle X^k  - X^0, \Dcal(X^k,\xi_k) \big\rangle \leq \beta  \|X^k - X^0\|_{\Fsf}^2  + \frac{1}{\beta}  \|\rgrad\,f(X^k)\|_{\Fsf}^2.
\ee 
Plugging \eqref{equ:lemma:xk:recursion:cc5} and \eqref{equ:lemma:xk:recursion:cc6} into \eqref{equ:lemma:xk:recursion:cc4}, we obtain \eqref{equ:theorem:xk:recursion}.

Finally, using \eqref{lemma:fX:descent:rsvrg:a1}, \eqref{equ:theorem:xk:recursion} and the fact that $\xi_k$ is independent of each other, considering the sequences $\{\abf_k\}, \{\bbf_k\}, \{\fbf_k\}$ in  Lemma \ref{lemma:recursion}  as $\{\Ebb_{\xi_{[K-1]}}\big[\| \rgrad\,f(X^k) \|_{\Fsf}^2 \big]\}$, $\{ \Ebb_{\xi_{[K-1]}}  \big[\|X^{k} - X^0\|_{\Fsf}^2\big] \}$ and $\{ \Ebb_{\xi_{[K-1]}}[\mF(X^{k})]\}$, respectively, we finally establish \eqref{equ:f:x:K:expectation}. 
\end{proof}



For fixed $s$, \eqref{equ:Delta:k} implies that
$\frac{\tau}{2} \big(2 \nu - \hat L  \tau \big) = \Delta_{K-1} >  \cdots > \Delta_{0}$. Thus
\[
\Delta_{\min} \coloneqq \min\limits_{s,k} \Delta_{s,k} = \min\limits_{s} \Delta_{s,0}.
\]
We now bound the variance of $\rgrad\, f(X_r)$ in expectation.
\begin{lemma} \label{theorem:svrg:grad:norm}
Consider Algorithm \ref{alg:usvrg}. Suppose that $\Delta_{\min} > 0$. Let $p_{s,k} = \frac{\Delta_{s,k}}{\sum_{k=0}^{K-1} \Delta_{s,k}},$ $k = 0, \ldots, K-1$ and $p_{s,K} = 0$. We have
\begin{align}\label{equ:theorem:svrg:grad:norm:00}
  \Ebb[\|\rgrad\, f(X_r)\|^2] \leq \frac{1}{SK\Delta_{\min}} \left( f(X^{0,0}) - f(X^*)\right)\!,
\end{align}
where $X^*$ is the optimal solution of problem \eqref{prob:manifold}, and the expectation is taken over $\xi_{[S-1], [K-1]} = (\xi_{0, [K-1]}, \ldots, \xi_{S-1, [K-1]})$ and the randomness of $r$.
\end{lemma}

\begin{proof}
First, note that $p_{s,K} = 0$, we have
\begin{align}  \label{equ:theorem:svrg:xrs}
\Ebb_{\xi_{s,[K-1]}} [\| \rgrad\, f(X_r^s)\|_{\Fsf}^2]  ={} & \sum_{k=0}^{K-1}  p_{s,k}\, \Ebb_{\xi_{[K-1]}} \big[\| \rgrad\,f(X^{s,k}) \|_{\Fsf}^2 \big]\nn \\
\leq{}&  \frac{f(X^{s,0}) - \Ebb_{\xi_{[K-1]}}[\mF(X^{s, K})]}{\sum_{k=0}^{K-1} \Delta_{s,k}}
\leq{} \frac{f(X^{s,0}) - \Ebb_{\xi_{[K-1]}}[\mF(X^{s, K})]}{K \Delta_{\min}}, 
\end{align}
where the first inequality is due to \eqref{equ:f:x:K:expectation}. 
Further note that
$$\Ebb[\|\rgrad\, f(X_r)\|^2] = \frac{1}{S} \sum_{s = 0}^{S-1}  \Ebb_{\xi_s, [K-1]} [\| \rgrad\, f(X_r^s)\|_{\Fsf}^2],$$ we know from \eqref{equ:theorem:svrg:xrs} and $X^{s+1,0} = X^{s,K}$ that 
\rblack{$\Ebb[\|\rgrad\, f(X_r)\|^2] \leq \frac{1}{SK\Delta_{\min}} \big( f(X^{s,0}) - \Ebb\big[f(X^{S,K})\big] \big),$}
which, together with $\Ebb\big[f(X^{S,K})\big] \geq f(X^*)$ implies \eqref{equ:theorem:svrg:grad:norm:00}. 
\end{proof}

By choosing the parameters $K, |\Bsf|, \beta$ and $\tau_s$ carefully, we are ready to establish the iteration complexity result of S-SVRG.
\begin{theorem} \label{theorem:complexity}
Consider Algorithm \ref{alg:usvrg}. Given constants $ 0 \leq \mu \leq 2/3$ and $\kappa > 0$, we set the parameters as
\be  \label{equ:svrg:K:B:beta:tau}
K = \lceil (\kappa n)^{\frac{1}{3(1 - \mu)}} \rceil, \quad |\Bsf| = \lceil K^{2 - 3\mu}\rceil,  \quad \beta = \frac{\sqrt{\tilde{L}} L}{\nu} K^{\mu - 1},  \quad \tau_s \equiv  \frac{c \nu}{\sqrt{\tilde{L}} L} K^{-\mu},
\ee
where the constant $c \in (0,1)$ satisfies
\be\label{equ:c} 
\frac{\hat L}{\sqrt{\tilde{L}} L} \exp(c^2 + 2c)c\leq 1.
\ee
To achieve a stochastic $\epsilon$-stationary point (defined in Definition \ref{def:stoc-eps-stat-point}) of problem \eqref{prob:manifold}, the IFO-calls and RO(retraction oracle)-calls complexities are \rblack{$O({n^{2/3}}/{\epsilon}+ n)$ and $O({n^{\frac{\mu}{3(1 - \mu)}}}/{\epsilon} + n^{\frac{1}{3(1-\mu)}})$}, respectively.
In particular, when $\mu = 0$, namely, $|\Bsf| = \lceil (\kappa n)^{2/3}\rceil$, the IFO-calls and RO-calls complexities become \rblack{$O({n^{2/3}}/{\epsilon} + n)$ and $O(1/{\epsilon} + n^{1/3})$}, respectively.
\end{theorem}
\begin{proof}
Since all $\tau_s$ are the same, again we drop the subscript $s$ for simplicity.

We first give the lower bound of  $\Delta_{\min}$.
Note that  $\tilde{L} \geq L_1^2 \geq 1$ and $0 \leq \mu \leq 2/3$, with the choices of $\tau$ and $|\Bsf|$ in \eqref{equ:svrg:K:B:beta:tau}, we have
\be \label{equ:usvrg:tau:beta:000}
\frac{\tilde{L} L^2 \tau^2}{\nu^2 |\Bsf|}   \leq c^2 K^{\mu - 2},\
1 + \frac{2}{\tilde{L} \beta \tau}  \leq 1 + \frac{2K}{c}.
\ee
By the first inequality in  \eqref{equ:usvrg:tau:beta:000}
 and the choice of $\beta$ in \eqref{equ:svrg:K:B:beta:tau}, we have from \eqref{equ:gamma:k} that 
\be\label{equ:usvrg:tau:beta:001}
\Gamma_0  \leq  \frac{\exp(c^2 + 2c) - 1}{c^2 + 2c} K.
\ee
Plugging \eqref{equ:usvrg:tau:beta:000} and \eqref{equ:usvrg:tau:beta:001}  into \eqref{equ:Delta:k} with $k = 0$,  and using the choice of $\tau$ in \eqref{equ:svrg:K:B:beta:tau}, we have
\begin{align} \label{equ:usvrg:tau:beta:Delta0:divide:tau:final}
 \frac{\Delta_0}{\tau}
\geq{}  \nu - \frac{\nu}{2} \frac{\hat L}{\sqrt{\tilde{L}} L} \exp(c^2 + 2c)c  \geq \frac{\nu}{2},
\end{align}
where the last inequality is due to \eqref{equ:c}. Note that $\Delta_{\min} = \Delta_0$, we thus have
\be\label{equ:lower:bound:Delta:min}
\Delta_{\min} \geq  \nu\tau/2.
\ee

Second, by choosing
\be\label{equ:theorem:complexity:S}
 S = \left \lceil  \frac{2(f(X^{0,0}) - f(X^*))}{\nu \epsilon}\cdot\frac{1}{K \tau}\right \rceil,
\ee
we know from \eqref{equ:lower:bound:Delta:min} and \eqref{equ:theorem:svrg:grad:norm:00} that
$\Ebb[\|\rgrad\, f(X_r)\|^2] \leq \epsilon.$
From \eqref{equ:svrg:K:B:beta:tau} we have $K \tau  \geq \frac{c \nu \kappa^{1/3}}{\sqrt{\tilde{L}} L}  n^{1/3}$, which together with \eqref{equ:theorem:complexity:S} gives
\be\label{equ:theorem:complexity:S:2}
S \leq  \frac{c_1 n^{-\frac13}}{\epsilon} + 1,
\ee
where $c_1 = \frac{2\sqrt{\tilde{L} L}(f(X^{0,0}) - f(X^*))}{c\nu^2 \kappa ^{1/3}}$. Using \eqref{equ:svrg:K:B:beta:tau} \rblack{and noting $|\Bsf|\leq 2K^{2 - 3\mu}$}, we have
 \be\label{equ:theorem:complexity:n:2KB}
n +  2K |\Bsf| \leq \rblack{n + 2K^{2 - 2\mu}} \leq c_2 n,
\ee
where \rblack{$c_2 = 1 + 4\max\{1, 4\kappa^{2/3}\}$}, and
\be\label{equ:theorem:complexity:K}
K \leq (\kappa n)^{\frac{1}{3(1 - \mu)}} + 1 \leq \big(\kappa^{\frac{1}{3(1 - \mu)}}  + 1\big)n^{\frac{1}{3(1 - \mu)}}.
\ee
Thus it follows from \eqref{equ:theorem:complexity:S:2} and \eqref{equ:theorem:complexity:n:2KB} that $\# \mbox{IFO-calls}=   S(n +  2K |\Bsf|) = O\big(n^{2/3}/\epsilon + n\big)$. Similarly, from \eqref{equ:theorem:complexity:S:2} and \eqref{equ:theorem:complexity:K}, we have $\# \mbox{RO-calls} = 2 SK = O\big(n^{\frac{\mu}{3(1 - \mu)}}/\epsilon + n^{\frac{1}{3(1-\mu)}}\big)$.
\end{proof}

\begin{remark}
Note that if we set the stochastic gradient $\Gcal(X^{k}, \xi_{k})$ to the exact gradient $\nabla f(X^k)$, then Algorithm \ref{alg:usvrg} reduces to the deterministic gradient descent method on manifold, and the complexity becomes $O({n}/{\epsilon})$, which was recently established in \cite{boumal2016global}. 
\end{remark}

\subsection{A S-SGD method}\label{subsection:usgd}
As a by-product of S-SVRG, we can give a vector transport-free SGD with general retraction for solving \eqref{prob:manifold} as in Algorithm \ref{alg:frsgd}. The stochastic event in Line 4 is denoted by $\xi_j$. Note that $\orth$ is compact, following the proof of Lemma \ref{lemma:usvrg:unbiased:bound_varaince}, we can also show that $$\Ebb_{\xi_k} [\Dcal(X^k, \xi_k) - \rgrad\, f(X^k)] = 0 \mbox{ and }  \Ebb_{\xi_k} \big[ \|\Dcal(X^k, \xi_k) - \rgrad\, f(X^k)\|_{\Fsf}^2 \big] \leq \sigma^2$$
for some constant $\sigma$. Using the similar techniques as in \cite{ghadimi2013stochastic}, we can show in Theorem \ref{theorem:sg:complexity} that the IFO-calls and RO-calls complexities of S-SGD are both $O({1}/{\epsilon^2})$.
For the sake of brevity, we omit the proof here. 

\begin{algorithm}[htbp]
\caption{A S-SGD for problem \eqref{prob:manifold}}\label{alg:frsgd}
     \begin{algorithmic}[1]
 \STATE{Given $X^0\in \orth$,  maximal iteration number $N$, \rblack{step sizes} $\{\tau_j\}_{j \geq 0}$, and probability mass function $\Pbb_{\bar j}(\cdot)$ supported on $\{0, \ldots, N-1\}$.}
\STATE{Generate the random variable $\bar j$ according to $\Pbb_{\bar j}(\cdot)$.}
 \FOR{$j = 0, \ldots, \bar j-1$}
 \STATE Pick a random $i \in \{1, \ldots, n\}$ uniformly and compute the stochastic {\it Euclidean} gradient $\Gcal(X^{j}, \xi_j)$ as
 $\Gcal(X^{j}, \xi_j) =  \nabla f_i(X^{j}).$
\STATE Compute the stochastic {\it Riemannian} gradient  as 
$\Dcal(X^{j}, \xi_j) = \Dbf_{\rho}\big(X^{j}, \Gcal(X^{j}, \xi_j)\big).$
 \STATE Update  $X^{j+1} = \Rscr\big({X^j}, -\tau_j \Dcal(X^j,\xi_j)\big).
$ 
 \ENDFOR
 \RETURN{$X^{\bar j}$.}
\end{algorithmic}
\end{algorithm}

\begin{theorem}\label{theorem:sg:complexity}
Let the \rblack{step sizes} $\{\tau_j\}$ be chosen as \rblack{$\tau_j \equiv \min\left\{\frac{\nu}{\hat L}, \frac{\tilde {D}}{\sigma \sqrt{N}}\right\}$}, where $\tilde {D} >0$ is some constant. Suppose that the probability mass function $\Pbb_{\bar j}(\cdot)$ in Algorithm \ref{alg:frsgd} is chosen as $\Pbb_{\bar j}(j) = {1}/{N}$, $j = 0, \ldots, N-1$. It holds that 
\be\label{equ:corollary:frsg:complexity}
\frac{1}{\hat L }\Ebb \left[ \| \rgrad\, f(X^{\bar j}) \|_{\Fsf}^2 \right]
 \leq  \frac{D_{f, \Rscr}^2}{N}  \left(\frac{\hat L}{\nu} + \frac{1}{T_{\Rscr}}\right)+ \left( \tilde D + \frac{D_{f, \Rscr}^2}{\tilde D}\right)\frac{\sigma}{\sqrt{N}},
\ee
where $D_{f, \Rscr} = [2(f(X^0) - f(X^*))/\hat L]^{\frac12}$ and the expectation is taken with respect to $\bar j$ and $\xi_{[N-1]} = (\xi_0, \xi_1, \ldots, \xi_{N-1})$. Moreover, \eqref{equ:corollary:frsg:complexity} indicates that the IFO-calls and RO-calls complexities of Algorithm \ref{alg:frsgd} for achieving a stochastic $\epsilon$-stationary point of problem \eqref{prob:manifold} are both $O({1}/{\epsilon^2})$.
\end{theorem}

\subsection{Local linear convergence of S-SVRG} \label{subsection:linear} 

We first establish the local linear convergence of S-SVRG by assuming that the \L ojasiewicz inequality holds, and then we prove that it holds with high probability for low-rank matrix completion problem. 

\begin{assumption}[\L ojasiewicz Inequality]\label{assumption:L}
For any stationary point $\bar X \in \orth$ of problem \eqref{prob:manifold}, there exist constants $\delta > 0$ and $\alpha > 0$ such that for all $\|X - \bar X\|_{\Fsf} \leq \delta$, it holds that
\be \label{equ:L}
|f(X) - f(\bar X)|^{1/2} \leq \alpha \|\rgrad\, f(X)\|_{\Fsf}.
\ee
\end{assumption}

\begin{theorem}\label{lemma:local:linear}
Assume Assumption \ref{assumption:L} holds. Consider Algorithm \ref{alg:usvrg} with ``$X^{s+1,0} \coloneqq X^{s,K}$'' in Line 11 replaced by ``$X^{s+1,0} \coloneqq X^{s}_r$.'' Suppose that the sequence $\{X^s_r\}$ converges to a stationary point $\bar X$ and suppose that all the iterate points lie in the set $\{X: \|X - \bar X\|_{\Fsf} \leq \delta\}.$ We choose the parameters according to  \eqref{equ:svrg:K:B:beta:tau} and choose the probability
\be\label{equ:linear:convergence:psk}
p_{s,k} =
\begin{cases}
  {\Delta_{s,k}}/{(\alpha^2 + \sum_{k=0}^{K-1}\Delta_{s,k})}, & k = 0, \ldots, K-1, \\[2pt]
   {\alpha^2}/{(\alpha^2 + \sum_{k=0}^{K-1}\Delta_{s,k})}, & k = K.
\end{cases}
\ee
It holds that $\{X^s_r\}$ converges to $\bar X$ linearly in expectation, i.e.,
\be\label{equ:lemma:local:linear}
\Ebb_{\xi_{s,[K-1]}}[f(X^{s}_r) - f(\bar X)] \leq  \frac{2 \sqrt{\tilde{L}}L \alpha^2}{2 \sqrt{\tilde{L}}L \alpha^2 + c  \nu^2 (\kappa n)^{\frac13} } (\rblack{f(X^{s-1}_r)}  - f(\bar X)).
\ee
\end{theorem}
\begin{proof}
First, by the \L ojasiewicz  inequality \eqref{equ:L}, we obtain from  \eqref{equ:f:x:K:expectation} that
\begin{align} \label{equ:lemma:local:linear:00}
\Ebb_{\xi_{s, [K-1]}}[\mF(X^{s, K})]
\leq{}& f(X^{s,0}) - \frac{1}{\alpha^2} \sum_{k=0}^{K-1}\Delta_k \Ebb_{\xi_{[K-1]}}[f(X^{s,k}) - f(\bar X)].
\end{align}
Note that $\Ebb_{\xi_{s,[K-1]}}[f(X^{s}_r) - f(\bar X)] = \sum_{k = 0}^K p_k \Ebb_{\xi_{s,[K-1]}} [f(X^{s,k}) - f(\bar X)]$. From \eqref{equ:lemma:local:linear:00} and \eqref{equ:linear:convergence:psk}, we have
\begin{align} \label{equ:lemma:local:linear:11}
\Ebb_{\xi_{s,[K-1]}}[f(X^{s}_r) - f(\bar X)]
\leq{}&\frac{\alpha^2}{\alpha^2 +  \sum_{k=0}^{K-1}\Delta_{s,k}} (f(X^{s,0})  - f(\bar X)).
\end{align}
From \eqref{equ:svrg:K:B:beta:tau}, we see that
\[
\sum_{k=0}^{K-1}\Delta_{s,k} \geq \rblack{\frac{K \nu \tau}{2}} \geq \frac{c \nu^2}{\rblack{2}\sqrt{\tilde{L}} L} (\kappa n)^{\frac13},
\]
which together with \eqref{equ:lemma:local:linear:11} and $X^{s,0} = X_r^{s-1}$ implies \eqref{equ:lemma:local:linear}.
\end{proof} 

There are some existing works which consider the linear convergence of the SVRG on manifold. In \cite{zhang2016fast}, Zhang, Reddi and Sra established the linear convergence of RSVRG for geodesically convex function. However, this result is trivial because every smooth geodescially convex function on a compact Riemannian manifold is a constant (see \cite{bishop1969manifolds}). Moreover,  \cite{zhang2016fast} also established the linear convergence result under the assumption that $f(X)$ is globally
$\tau$-gradient dominated, i.e., there exists $\tau>0$ such that
\be \label{equ:gradient:dominated}
f(X) - f(X^*) \leq \tau \|\rgrad\, f(X)\|_{\Fsf}^2, \quad \forall X \in \Mcal,
\ee
where $X^*$ is the optimal solution. However, it should be noted that \eqref{equ:gradient:dominated} is very difficult to be verified because $X^*$ is unknown. Kasai, Sato and Mishra \cite{kasai2016riemannian} established the local linear convergence of R-SVRG under the assumption that the sequence converges to a non-degenerate local minimizer at which the Riemannian Hessian is positive definite. Xu and Ke \cite{xu2016stochastic} showed the linear convergence of their SVRRG for eigenvalue problem where they assume that the initial point is sufficiently close to the optimal solution. Note that none of the above three works gave a nontrivial example that satisfies their corresponding assumptions. Besides, Wu \cite{weijie2016quadratic} established the local linear convergence of Stiefel-SVRG by using the fact that the \L ojasiewicz inequality holds for PCA problem \cite{liu2016quadratic}.

In the following we show that the \L ojasiewicz inequality \eqref{equ:L} holds locally with high probability for matrix completion problems on Grassmann manifold. 

Given a rank $r$ matrix $M \in \Rbb^{m \times n}$ with $m \geq n$ and $M = U\Sigma V^{\Tsf}$, where $U^{\Tsf} U = m I_r$, $V^{\Tsf} V = n I_r$ (note that this can be obtained by the SVD of $M$). The matrix completion problem aims to recover $M$ by partial observations on a subset $\Omega \subset \{1, \ldots, m\} \times \{1, \ldots, n\}$. As is done in \cite{keshavan2010matrix}, we define
\[
F(W, Z) \coloneqq{} \min_{S \in \Rbb^{r \times r}} \Fcal (W, Z, S) \quad \mbox{and} \quad
\Fcal(W,Z,S) \coloneqq \frac12 \|\proj_{\Omega} (M - WSZ^{\Tsf})\|_{\Fsf}^2, \nn
\]
where $W \in \Rbb^{m \times r}$ and $Z \in \Rbb^{n \times r}$ satisfy $W^{\Tsf} W = m I_r$ and $Z^{\Tsf} Z = n I_r$.
Consider
\begin{align}\label{equ:def:tF}
\tF(W, Z) 
={} & F(W,Z) + \varrho \sum_{i = 1}^m G_1 \left( \frac{\|W^{(i)}\|^2}{3 \mu_0 r}\right) + \varrho \sum_{j = 1}^n G_1 \left( \frac{\|Z^{(j)}\|^2}{3 \mu_0 r}\right)\!, 
\end{align}
where the parameter $\varrho>0$, \rblack{$\mu_0$ is the incoherence parameter (see, e.g., \cite{keshavan2010matrix}) of $M$}, $W^{(i)}$ and $Z^{(j)}$ are respectively the $i$th and $j$th columns of $W^{\Tsf}$ and $Z^{\Tsf}$, and 
\be\label{equ:G1}
G_1(z) = \begin{cases}0, & \mbox{if} \  z \leq 1, \\ e^{(z-1)^2} - 1, & \mbox{if}\ z \geq 1. \end{cases}
\ee
The regularized matrix completion problem is formulated as follows \cite{keshavan2010matrix}:
\be \label{prob:MC}
\min\, \tF(\ubf), \quad \st   \quad \ubf \in \mathsf{M}(m,n),
\ee
where $\ubf = (W,Z)$ and the Cartesian product Grassmann manifold $\mathsf{M}(m,n) =  \{(W \in \Rbb^{m \times r}, Z \in \Rbb^{n \times r}) \colon W^{\Tsf}W = mI_r, Z^{\Tsf} Z = nI_r\}$. 
Note that $\tF(\ubf) = \tF(W,Z) \equiv \tF(WQ_1,ZQ_2)$ for any \rblack{$Q_1, Q_2 \in \mathsf{St}_r$}. Define
$$\Kcal(\mu') = \left\{(W,Z):  \|W^{(i)}\|^2 \leq \mu' r,   \|Z^{(j)}\|^2 \leq \mu' r \quad \forall i \in \{1, \ldots, m\}, j \in \{1, \ldots, n\}\right\}.
$$
Denote $\ubf^* = (U,V)$. If $\ubf^* \in \mathsf{M}(m,n) \cap\Kcal(3 \mu_0)$, \rblack{from \eqref{equ:G1} and \eqref{equ:def:tF}, we know $\tF(u) \equiv 0$, which means that} $\ubf^*$ is the optimal solution of \eqref{prob:MC}. We now give the \L ojasiewicz inequality result of \eqref{prob:MC}, whose proof is relegated to
Appendix \ref{section:L:Grassmann}.

\begin{theorem}\label{theorem:L:inequality:MC}
For $M = U \Sigma V^{\Tsf}$ with $\ubf^* := (U,V)\in \mathsf{M}(m,n) \cap \Kcal(3\mu_0)$. For any $\ubf \in \mathsf{M}(m,n) \cap \Kcal(4\mu_0)$ and $d(\ubf, \ubf^*) \leq \delta,$ it holds with probability at least
$1 - 1/n^4$ that
\be \label{theorem:L:inequality:MC:00}
\left(\tF(\ubf) - \tF(\ubf^*)\right)^{\frac12}  \leq  \left(\frac{m  \Sigma_{\max}^2 + \frac{14 e^{\frac19}}{9 \mu_0 r}  \rho }{C\epsilon^2 \Sigma_{\max}^4}\right)^{\frac12}
\|\rgrad \tF(\ubf)\|_{\Fsf},
\ee
where the probability is taken with respect to the uniformly random subset $\Omega$ and $\delta, C, \epsilon$ are some constants. For the definition of $d(\cdot, \cdot)$ and the specific conditions on $\delta, C$ and $\epsilon$, see equation (22) and Lemma 6.5 in \cite{keshavan2010matrix}.
\end{theorem}

\section{Extensions} \label{section:extensions}
In this section, we extend the S-SVRG for problem \eqref{prob:manifold} to more general manifold $\Mcal$. Similar to Assumption \ref{assumption_lipschitz}, throughout this section, we assume that $f_i(X)$ is differentiable and $\nabla f_i(X)$ is $L$-Lipschitz continuous over $\Mcal$. The corresponding extension of S-SGD is also possible, and we omit the details for brevity. 

\subsection{Some special manifolds related to $\orth$}\label{subsection:special:manifolds}
The S-SVRG method can be naturally extended to optimization with the generalized orthogonality constraints
\be \label{equ:general:orth}
\{(X_1\in\Rbb^{d_1 \times r_1}, \ldots, X_p\in\Rbb^{d_p \times r_p}) \colon  X_i^{\Tsf} M_i X_i = I_{r_i},\, i = 1, \ldots, p\}
\ee
where $M_i$ is symmetric positive definite. Besides, the low-rank elliptope
$$\{X \in \Rbb^{m \times m} \colon \diag(X) = \1_m,\  X = X^{\Tsf} \succeq 0,\ \mathrm{rank}(X) \leq r \leq m\},$$
where $\diag(X^{\Tsf} X)$ is the diagonal vector of $X^{\Tsf} X$. The low-rank spectrahedron
$$\{X \in \Rbb^{m \times m} \colon \mtr(X) = 1,\  X \succeq 0,\ \mathrm{rank}(X) \leq r\leq m\}, $$
and  the oblique manifold
$$\{X \in \Rbb^{m \times p} \colon \mathrm{diag}(X^{\Tsf} X) = \1_p\},$$
where $\1_p \in \Rbb^{p}$ is the all-one vector,
can be seen as special cases of \eqref{equ:general:orth} by some simple transformations. Specifically, letting $X = H^{\Tsf} H$ with $H \in \Rbb^{r \times m}$, the low-rank elliptope and the low-rank spectrahedron can be represented as
$\{H \in \Rbb^{r \times m} \colon \|H_i\|_{\Fsf} = 1, i = 1, \ldots, m\}$ and $\{H \in \Rbb^{r \times m} \colon \|H\|_{\Fsf} = 1\}$, respectively;
the oblique manifold is equivalent to multiple sphere constraints as
$\{X \in \Rbb^{m \times p} \colon \|X_i\|_{\Fsf} = 1, i = 1, \ldots, p\},$
where $X_i$ is the $i$th column of $X$.

\subsection{More general manifolds} \label{subsection:general:manifolds}
S-SVRG (Algorithm \ref{alg:usvrg}) is still well-defined if $\orth$ is replaced by a general Riemannian manifold $\Mcal$.
Since the tangent space is linear, we can still establish \eqref{equ:svrg:Dcal:unbiased} and \eqref{equ:svrg:Dcal:variance}. To obtain the complexity results, we need the following assumption.

\begin{assumption}\label{assumption:general:manifold:retraction}
Consider the retraction $\rblack{\Rscr(t)}$ with $\rblack{\Rscr(0)} = X$ on $\Mcal$. We assume that there exists some positive constants $L^{\Mcal}_{1} \geq 1$, $L_2^{\Mcal}$ \rblack{and the universal positive constant $T_{\Rscr}^{\Mcal}$} such that
\begin{align}
\| \rblack{\Rscr(t)} - \rblack{\Rscr(0)}\|_{\Fsf} \leq L_1^{\Mcal} t \|\rblack{\Rscr'(0)}\|_{\Fsf}, \label{equ:retraction:b1:general}\\
\|\rblack{\Rscr(t)} - \rblack{\Rscr(0)} - t \rblack{\Rscr'(0)}\|_{\Fsf} \leq L_2^{\Mcal} t^2 \|\rblack{\Rscr'(0)}\|_{\Fsf}^2,\label{equ:retraction:b2:general}
\end{align}
for any $t \in [0, T_{\Rscr}^{\Mcal}]$.
\end{assumption}

If further assuming that $\nabla f_i(X)$ is bounded on $\Mcal$, then we can obtain the same complexity result as shown in Theorem \ref{theorem:complexity}. The proof is given in Appendix \ref{section:proof:bounded:g}.

\begin{theorem}\label{theorem:complexity:bounded:g}
Consider Algorithm \ref{alg:usvrg} for problem \eqref{prob:manifold-0}. Suppose that  Assumption  \ref{assumption:general:manifold:retraction} holds. Moreover, we assume that
\be \label{equ:g:bound}
\|\nabla f_i(X)\|_{\Fsf} \leq C^{\Mcal}, \quad \forall X \in \Mcal.
\ee
For any $0 \leq \mu \leq 2/3$ and $\kappa > 0$, we choose
\be\label{equ:K:svrg:bounded:g}
K = \lceil (\kappa n)^{\frac{1}{3(1 - \mu)}} \rceil, \  |\Bsf| = \lceil K^{2 - 3\mu}\rceil,  \  \beta = \frac{c}{b} K^{\mu - 1},  \  \rblack{\tau_s \equiv  b K^{-\mu}},
\ee
where $b = \min\Big\{\frac{c \nu}{\sqrt{\tilde{L}^{\Mcal}} L }, 1, T_{\Rscr}^{\Mcal}\Big\}$, $\tilde{L}^{\Mcal}  =  (L^{\Mcal}_1)^2 +  6L_1^{\Mcal} L_2^{\Mcal} C^{\Mcal}$ and the constant $c \in (0,1)$ satisfies
\be\label{equ:c:bounded:g} 
  \frac{\hat L^{\Mcal}}{\sqrt{\tilde{L}^{\Mcal}} L} \exp(c^2 + 2c)  c \leq 1. 
\ee
Let $p_{s,k} = \frac{\Delta_{s,k}}{\sum_{k=0}^{K-1} \Delta_{s,k}}, k = 0, \ldots, K-1$ and $p_{s,K} = 0$, where 
\rblack{
\begin{align}  \label{equ:Delta:bounded:g}
\frac{\Delta_{s,k}}{\tau_s} =  \nu -  \frac{\hat L^{\Mcal}}{2} \tau_s \left[ 1 +   \left( \frac{\tilde{L}_1^{\Mcal} + \tilde{L}_2^{\Mcal} K \tau_s}{\tilde{L}^{\Mcal}}  + \frac{2}{\tilde{L}^{\Mcal} \beta \tau_s}\right)  \frac{\tilde{L}^{\Mcal}L^2 \tau_s^2}{\nu^2 |\Bsf|}   {\Gamma}_{s,k}^{\Mcal}\right]
\end{align}
with
\be \label{equ:section:proof:bounded:g:gamma:k}
{\Gamma}_{s,k}^{\Mcal} = \Gamma\Big(2 \beta \tau_s +    \frac{\big(\tilde{L}_1^{\Mcal}  + \tilde{L}_2^{\Mcal} K \tau_s\big)L^2}{\nu^2 |\Bsf|} \tau_s^2 , K - k\Big)
\ee
 in which  $\tilde{L}^{\Mcal}_1  =  (L^{\Mcal}_1)^2$ and $\tilde{L}^{\Mcal}_2 = 6L_1^{\Mcal} L_2^{\Mcal} C^{\Mcal}.$}
To obtain a stochastic $\epsilon$-stationary point of \eqref{prob:manifold}, the IFO-calls and RO-calls complexities are \rblack{$O({n^{2/3}}/{\epsilon}+ n)$ and $O({n^{\frac{\mu}{3(1 - \mu)}}}/{\epsilon} + n^{\frac{1}{3(1-\mu)}})$,} respectively. In particular, when $\mu = 0$, i.e., $|\Bsf| = \lceil (\kappa n)^{2/3}\rceil$, the IFO-calls and RO-calls complexities become \rblack{$O({n^{2/3}}/{\epsilon} + n)$ and $O(1/{\epsilon} + n^{1/3})$}, respectively.
\end{theorem}


If the boundedness assumption \eqref{equ:g:bound} does not hold, as in \cite{boumal2016global}, we need the following assumption.
\begin{assumption}[\cite{boumal2016global}] \label{assumption:general:manifold:f}
For any $t \geq 0$,  we assume that there exits a universal positive constant $\hat L^{\Mcal}$ such  that 
\[
f(\rblack{\Rscr(t)})\leq f(X)  + t \langle \rgrad\, f(X), E \rangle_X + \frac{\hat L^{\Mcal}}{2} t^2 \|E\|_{\Fsf}^2
\]
for any $t \in [0, T_{\Rscr}^{\Mcal}]$. 
\end{assumption}

We thus can establish the iteration complexity results as follows. \rblack{The proof is given in Appendix \ref{section:proof:no:bounded:g}.}

\begin{theorem}\label{theorem:complexity:no:bounded:g}
Consider Algorithm \ref{alg:usvrg} for problem \eqref{prob:manifold-0}. Suppose that Assumption  \ref{assumption:general:manifold:retraction} and Assumption \ref{assumption:general:manifold:f} hold.
For any $0 \leq \theta \leq 1$ and $\kappa > 0,$ we choose
\be  \label{equ:K:svrg:no:bounded:g}
K = \lceil (\kappa n)^{\frac{1}{3 - 2\theta}}\rceil, \  |\Bsf| = \lceil K^{2 - 2\theta}\rceil, 
\ \beta = \frac{c}{K},  \ \rblack{\tau_s \equiv  b K^{-\theta},}
\ee
where \rblack{$b=\min\Big\{\frac{c \nu}{L_1^{\Mcal}L }, 1, T_{\Rscr}^{\Mcal}\Big\}$} and the constant $c \in (0,1)$ satisfies
\be\label{equ:c:no:bounded:g} 
 \frac{\hat L^{\Mcal}}{L_1^{\Mcal} L} \left(\frac{c+1}{c+2} \exp(c^2 + 2c) + \frac{1}{c+2}\right) c \leq 1.
\ee
Let $p_{s,k} = \frac{\Delta_{s,k}}{\sum_{k=0}^{K-1} \Delta_{s,k}}, k = 0, \ldots, K-1$ and $p_{s,K} = 0,$ where
\rblack{
\begin{align} \label{equ:Delta:no:bounded:g}
\frac{\Delta_{s,k}}{\tau_s} =  \nu - \frac{\hat L^{\Mcal}}{2}  \tau_s \left[1 + \left(1 + \beta^{-1} \right)  \frac{L^2 (L_1^{\Mcal})^2 \tau_s^2 }{\nu^2 |\Bsf|}  \Gamma_{s,k}^{\Mcal}\right]\!,
\end{align}
with $\Gamma_{s,k}^{\Mcal}  = \Gamma\Big(\beta  +   \big(1 +   \beta^{-1} \big) \frac{L^2 (L_1^{\Mcal})^2 \tau_s^2}{\nu^2 |\Bsf|},K-k\Big),$}
To obtain a stochastic $\epsilon$-stationary point of problem \eqref{prob:manifold}, the IFO-calls and RO-calls complexities are \rblack{$O\big({n^{\frac{2 -  \theta}{3 - 2 \theta }}}/{\epsilon} + n\big)$ and $O\big({n^{\frac{\theta}{3 - 2 \theta}}}/{\epsilon} + n^{\frac{1}{3 - 2\theta}}\big)$}, respectively. In particular, when $\theta = 0$, i.e., $|\Bsf| = \lceil (\kappa n)^{2/3}\rceil$, the IFO-calls and RO-calls complexities become \rblack{$O({n^{2/3}}/{\epsilon} + n)$ and $O(1/{\epsilon}+ n^{1/3})$}, respectively.
\end{theorem}





\section{A practical S-SVRG-BB algorithm}\label{section:usvrgBB}
One of the major issues in SGD is how to choose step size while running the algorithm. Recently, Tan \etal \cite{tan2016barzilai} proposed the SVRG-BB method, which incorporates the BB \rblack{step size} \cite{barzilai1988two} to SVRG. The numerical results showed that SVRG-BB performs comparably to SVRG with best-tuned \rblack{step sizes}. BB step size was also used in optimization on Riemannian manifold (see, e.g., \cite{iannazzo2015riemannian, jiang2015framework, wen2013feasible}). Motivated by these results, we propose to incorporate the BB step size to compute the step size $\tau_s$ in S-SVRG. Similar as \cite{wen2013feasible}, we define
$$\mathsf{S}^s =  X^{s} - X^{s - 1} \mbox{\quad and \quad } \mathsf{Y}^s = \rgrad\, f(X^s) - \rgrad\, f(X^{s-1}),$$
and compute the BB \rblack{step size} by
\be\label{equ:BB}
 \tau_s^{\LBB} = {\langle \mathsf{S}^s, \mathsf{S}^s\rangle}/{|\langle \mathsf{S}^s, \mathsf{Y}^s\rangle|}.
\ee
Given $0< \tau_{\min} < \tau_{\max}$, as done in \cite{raydan1997barzilai}, we provide safeguards for $\tau_s^{\LBB}$, namely, compute $\tau_s^{\LBB} = \max\{\tau_{\min},  \min\{\tau_s^{\LBB}, \tau_{\max}\}\}$. We set
\be  \label{equ:tau:BB:K}
\tau_s = {\tau_s^{\LBB}}/{K}.
\ee
We call Algorithm \ref{alg:usvrg} with $\tau_s$ computed by \eqref{equ:tau:BB:K} as S-SVRG-BB algorithm.  \rblack{By rewriting \eqref{equ:tau:BB:K} as $\tau_s = {c \nu}/{\sqrt{\tilde L} L}K^{-\mu}$ with $c = \rho K^{\mu - 1}$ in which 
\be\label{equ:USVRGBB:rho}
\rho \nu/{\sqrt{\tilde L} L} \in\left[\tau_{\min}, \tau_{\max}\right],
\ee  we immediately obtain the following result for S-SVRG-BB from  Theorem \ref{theorem:complexity}. }
\begin{corollary} 
\rblack{Consider S-SVRG-BB algorithm. Given constant $0 \leq \mu \leq 2/3$, we select $K$, $|\Bsf|$, $\beta$ and $\tau_s$ by \eqref{equ:svrg:K:B:beta:tau}, 
where the positive constant $\kappa$ is chosen such that $c = \rho K^{\mu - 1}$ ($\rho$ satisfies \eqref{equ:USVRGBB:rho}) lies in $(0,1)$ and satisfies \eqref{equ:c}.}
To obtain a stochastic $\epsilon$-stationary point of \eqref{prob:manifold}, the IFO-calls and RO-calls complexities are $O(n/\epsilon)$ and $O(1/\epsilon)$, respectively.
\end{corollary}

\section{Numerical results} \label{section:numerical}

In this section, we compare VR-PCA \cite{shamir2015fast}, R-SVRG \cite{kasai2016riemannian} and SVRRG \cite{xu2016stochastic}
with our S-SVRG and S-SVRG-BB for solving PCA and matrix completion (MC) problems. Note that for problem \eqref{prob:manifold}, SMART-SVRG \cite{aravkin2016smart} is a special case of S-SVRG where the retraction is the gradient projection.  If we restrict $f(X)$ to be a quadratic function (without linear term) and
use the retraction of QR factorization or polar decomposition, our S-SVRG with fixed \rblack{step size} becomes Stiefel-SVRG \cite{weijie2016quadratic}. We consider seven types of retractions \eqref{equ:retraction:qr}-\eqref{equ:retraction:gr} and we denote them by ``qr'', ``pd'', ``wy'', ``jd'',  \rblack{``gp'', ``exp''} and ``gr'', respectively (see Appendix \ref{section:retraction:stiefel} and \ref{section:retraction:grass} for details). For each test instance, we run each method 20 times from random initial points. For each run, the initial random number generator seeds for different methods are the same. We stop each method at the $s$th epoch when $\|\rgrad\,f(X^{s,0})\|_{\Fsf}\leq 10^{-6}$ or $s$ is larger than or equal to the maximal epoch number 200. We always choose the batchsize $|\Bsf| = 0.01 n$ and set the maximal inner iteration number as \rblack{$K = 5n$}. For all the aforementioned SVRG-type methods, as done in \cite{johnson2013accelerating}, we use $K$ iterations of the S-SGD method to improve the quality of the random initial points. For PCA problem, we set $\rho = 0$ since there always hold $X^{\Tsf} \Gcal(X^{s,k}, \xi_{s,k}) \equiv \Gcal(X^{s,k}, \xi_{s,k})^{\Tsf} X$ and $X^{\Tsf} \nabla f(X) \equiv \nabla f(X)^{\Tsf} X$ and thus $\Dcal(X^{s,k}, \xi_{s,k}) \equiv  \Dbf_{0}(X^{s,k}, \Gcal(X^{s,k}, \xi_{s,k}))$ and $\rgrad\,f(X^s) \equiv \Dbf_{0}(X^s, \nabla f(X^s))$. For MC problem, we shall choose specific $\rho$ for different methods since $\Dcal(X^{s,k}, \xi_{s,k}) $ depends on $\rho$ although we always have  $\rgrad\,f(X^s) \equiv \Dbf_{0}(X^s, \nabla f(X^s))$. \rblack{The parameters $\tau_{\min}$ and $\tau_{\max}$ are chosen to be $10^{-8}$ and $10^8$, respectively.}
\rblack{In our numerical experiments, we use a new function $\phi(t)$ in `jd' retraction \eqref{equ:retraction:jd} as}
\be
\rblack{\phi(t) = \begin{cases} t/2, &\mbox{if $t < 10^{-10}$},\\ 1/2, &  \mbox{if $t \geq 10^{-10}$.}\end{cases}} \nn
\ee
Note that such chosen $\phi(t)$ satisfies the condition \eqref{equ:jd:phi} and it emphasis the role of $X^{\Tsf} E$ in \eqref{equ:retraction:jd} when $t$ is small. 

Our codes were written in MATLAB (\rblack{Release 2016b}) and all the experiments were conducted in \rblack{Ubuntu 16.04 LTS} on \rblack{a Dell workstation with a 3.5-GHz Intel Xeon E3-1240 v5 processor with access to 32 GB of RAM.} 

\subsection{Principal component analysis}

Given the observation data matrix $A\in\Rbb^{d\times n}$, the PCA problem can be formulated as
\be \label{prob:pca}
\max_{X \in \Rbb^{d \times r}}\  \frac{1}{n} \mtr(X^{\Tsf} (A - \mathbf{\bar A})(A - \mathbf{\bar A})^{\Tsf} X), \quad \st  \quad X^{\Tsf} X = I_r,
\ee
where $\mathbf{\bar A} =  \bar A \1_d^{\Tsf}$ with $\bar A = \frac1n \sum_{i=1}^n  A_i$ and $A_i$ being the $i$th column of $A$. 
It is easy to see that \eqref{prob:pca} is a special case of \eqref{prob:manifold} with $f_i(X) = -\mtr(X^{\Tsf} (A_i - \bar A)  (A_i - \bar A)^{\Tsf} X)$.

In our experiments, we generated $A$ using to the following MATLAB code:
\begin{verbatim}
        temp = [1:d].^(0.618);  A = randn(d,N);
        A = bsxfun(@times,temp',A);  A = A./max(max(abs(A)));
\end{verbatim}
 \rblack{We set $d = 1,000$, $n = 10,000$ and consider three choices of $r$: 10, 20, 40.} We first compare S-SVRG with seven retractions and existing methods R-SVRG, VR-PCA and SVRRG. For these methods, we report their performance with best-tuned step sizes chosen from the set $\{1, 2, \ldots, 100\}$. The results are reported in Table \ref{table:pca:rsvrg:best:tuned}. In this table, ``retr.'' denotes the type of the retraction,  ``\,$\tau^*$'' denotes the best-tuned \rblack{step size}, ``epoch'' gives the minimal, average and maximal number of epoches and the standard deviation. The term ``\,$\mathrm{\overline{nrm}}$'' denotes the average $\mathsf{F}$-norm of Riemannian gradient at the point returned by each method while ``\,$\mathrm{\overline{err}}$'' is the average relative function value error to the optimal value. The average CPU time ``\,$\mathrm{\overline{t}}$'', evaluated by the tic-toc commands, is in seconds. We have the following observations from Table \ref{table:pca:rsvrg:best:tuned}. For the best-tuned \rblack{step sizes}, our S-SVRG is always faster than R-SVRG and SSVRG while the quality of the solution is similar. The performance of S-SVRG with the retraction of `gp', (i.e., SMART-SVRG), is better than VR-PCA which also adopts the retraction of `gp' but with additional twist procedures. In terms of the average number of epochs, S-SVRG with `gp' is always the worst one among the seven tested retractions. 

\begin{table}[!htbp]
\caption{Comparison of S-SVRG and existing methods for PCA instances: best-tuned \rblack{step size}, $d = 1000, n= 10000$. }\label{table:pca:rsvrg:best:tuned}
\vspace{2pt}
\centering
\begin{footnotesize} 
\rblack{
\begin{tabular}{|@{\hspace{0.9mm}}c@{\hspace{1mm}}c@{\hspace{0mm}}
|c@{\hspace{1.5mm}}c@{\hspace{1.5mm}}l@{\hspace{1.5mm}}c@{\hspace{1.5mm}}c@{\hspace{1.5mm}}r@{\hspace{0.2mm}}
c@{\hspace{2mm}}|c@{\hspace{1.5mm}}l@{\hspace{1.5mm}}c@{\hspace{1.5mm}}c@{\hspace{1.5mm}}r@{\hspace{0.9mm}}|
}
\hline
 && \multicolumn{6}{c}{Existing methods} && \multicolumn{5}{c|}{S-SVRG} \\
  \cline{3-9} \cline{10-14}
 \Gape[6pt]  retr. && method & $\tau^*$ & epoch & $\mathrm{\overline{nrm}}$ & $\mathrm{\overline{err}}$& $\mathrm{\overline{t}}$ &&
   $\tau^*$ & epoch & $\mathrm{\overline{nrm}}$ & $\mathrm{\overline{err}}$& $\mathrm{\overline{t}}$\\
\Xhline{0.6pt}
\hline\hline
&&\multicolumn{12}{@{}c|}{$r = 10$} \\ 
   exp &&  RSVRG & 1.2 & (39,\,52.9,\,74,\,8.2) & 9e-07 & 7e-11 & 25.2&& 1.2 & (41,\,52.9,\,72,\,8.8) & 9e-07 & 8e-11 & 11.2\\ 
    pd &&  SVRRG & 1.2 & (42,\,53.1,\,70,\,7.7) & 9e-07 & 7e-11 & 11.5&& 1.2 & (41,\,53.0,\,73,\,8.8) & 9e-07 & 7e-11 & 8.7\\ 
    gp && VR-PCA & 1.3 & (42,\,53.1,\,67,\,8.3) & 9e-07 & 7e-11 & 10.4&& 1.3 & (39,\,53.4,\,73,\,10.1) & 9e-07 & 7e-11 & 7.8\\ 
\hdashline[2pt/2pt]
    qr && --- & --- & --- & --- & --- & ---&& 1.1 & (32,\,52.9,\,70,\,10.9) & 9e-07 & 1e-10 & 8.2\\ 
    wy && --- & --- & --- & --- & --- & ---&& 1.2 & (41,\,52.8,\,72,\,8.8) & 9e-07 & 8e-11 & 9.8\\ 
    jd && --- & --- & --- & --- & --- & ---&& 1.2 & (41,\,52.8,\,72,\,8.8) & 9e-07 & 8e-11 & 9.9\\ 
   gr && --- & --- & --- & --- & --- & ---&& 0.6 & (48,\,55.0,\,76,\,6.8) & 9e-07 & 8e-11 & 8.1\\ 
\hline \hline 
&&\multicolumn{12}{@{}c|}{$r = 20$} \\ 
   exp &&  RSVRG & 1.3 & (73,\,110.5,\,159,\,19.4) & 9e-07 & 5e-11 & 100.0&& 1.4 & (67,\,106.2,\,156,\,18.5) & 9e-07 & 4e-11 & 37.3\\ 
    pd &&  SVRRG & 1.3 & (71,\,113.8,\,172,\,22.4) & 9e-07 & 4e-11 & 35.0&& 1.3 & (80,\,107.2,\,128,\,12.6) & 9e-07 & 5e-11 & 23.3\\ 
    gp && VR-PCA & 1.4 & (80,\,109.3,\,153,\,18.9) & 9e-07 & 5e-11 & 30.8&& 1.4 & (70,\,107.3,\,139,\,17.3) & 9e-07 & 5e-11 & 20.6\\ 
\hdashline[2pt/2pt]
    qr && --- & --- & --- & --- & --- & ---&& 1.3 & (71,\,106.3,\,138,\,18.0) & 9e-07 & 5e-11 & 21.6\\ 
    wy && --- & --- & --- & --- & --- & ---&& 1.4 & (66,\,106.1,\,151,\,17.9) & 9e-07 & 4e-11 & 27.9\\ 
    jd && --- & --- & --- & --- & --- & ---&& 1.4 & (66,\,106.1,\,151,\,17.9) & 9e-07 & 5e-11 & 24.4\\ 
   gr && --- & --- & --- & --- & --- & ---&& 0.7 & (54,\,101.8,\,139,\,24.0) & 9e-07 & 5e-11 & 18.5\\ 
\hline \hline 
&&\multicolumn{12}{@{}c|}{$r = 40$}  \\ 
   exp &&  RSVRG & 1.3 & (82,\,104.7,\,127,\,13.4) & 9e-07 & 2e-11 & 182.8&& 1.4 & (75,\,104.2,\,126,\,13.6) & 9e-07 & 2e-11 & 63.0\\ 
    pd &&  SVRRG & 1.3 & (85,\,103.8,\,131,\,14.0) & 9e-07 & 2e-11 & 46.1&& 1.4 & (71,\,101.8,\,125,\,16.6) & 9e-07 & 2e-11 & 34.9\\ 
    gp && VR-PCA & 1.4 & (73,\,106.3,\,164,\,23.7) & 9e-07 & 3e-11 & 49.0&& 1.2 & (78,\,111.2,\,138,\,18.9) & 9e-07 & 4e-11 & 35.7\\ 
\hdashline[2pt/2pt]
    qr && --- & --- & --- & --- & --- & ---&& 1.3 & (80,\,98.2,\,129,\,14.0) & 9e-07 & 3e-11 & 32.6\\ 
    wy && --- & --- & --- & --- & --- & ---&& 1.4 & (74,\,103.9,\,126,\,13.9) & 9e-07 & 2e-11 & 46.5\\ 
    jd && --- & --- & --- & --- & --- & ---&& 1.4 & (74,\,103.9,\,126,\,13.9) & 9e-07 & 3e-11 & 38.4\\ 
   gr && --- & --- & --- & --- & --- & ---&& 0.7 & (72,\,103.1,\,137,\,18.8) & 9e-07 & 3e-11 & 27.9\\ 
\hline \hline 
&&\multicolumn{12}{@{}c|}{$r = 60$} \\ 
   exp &&  RSVRG & 1.2 & (78,\,100.7,\,148,\,17.0) & 9e-07 & 2e-11 & 281.2&& 1.3 & (65,\,88.3,\,109,\,12.8) & 9e-07 & 3e-11 & 82.7\\ 
    pd &&  SVRRG & 1.2 & (74,\,95.4,\,120,\,12.6) & 9e-07 & 2e-11 & 63.9&& 1.3 & (61,\,86.3,\,109,\,14.0) & 9e-07 & 3e-11 & 40.4\\ 
    gp && VR-PCA & 1.4 & (67,\,95.0,\,143,\,19.4) & 9e-07 & 2e-11 & 66.3&& 1.3 & (70,\,99.0,\,130,\,17.2) & 9e-07 & 2e-11 & 41.4\\ 
\hdashline[2pt/2pt]
    qr && --- & --- & --- & --- & --- & ---&& 1.4 & (77,\,95.5,\,118,\,11.9) & 9e-07 & 2e-11 & 44.2\\ 
    wy && --- & --- & --- & --- & --- & ---&& 1.3 & (65,\,87.7,\,109,\,13.2) & 9e-07 & 2e-11 & 68.7\\ 
    jd && --- & --- & --- & --- & --- & ---&& 1.3 & (65,\,87.7,\,109,\,13.2) & 9e-07 & 3e-11 & 49.9\\ 
   gr && --- & --- & --- & --- & --- & ---&& 0.7 & (64,\,88.7,\,113,\,12.7) & 9e-07 & 2e-11 & 35.6\\ 
\hline \hline \Xhline{0.6pt}
\end{tabular}
}
\end{footnotesize}
\end{table}


To investigate the efficiency of S-SVRG-BB, we compare the performance of S-SVRG-BB and S-SVRG with best-tuned \rblack{step sizes}.
The comparison results are reported in Table \ref{table:pca:rsvrg:bb}. From this table, we see that the S-SVRG-BB with the retraction `jd' performs best, while the S-SVRG-BB with the retraction `exp' performs worst, but all of them are comparable with S-SVRG with best-tuned step sizes.

\begin{table}[!htbp]
\caption{Comparison of S-SVRG with best-tuned \rblack{step sizes} and S-SVRG-BB for PCA instances: $d = 1000, n = 10000$.}\label{table:pca:rsvrg:bb}
\vspace{2pt}
\centering
\begin{footnotesize} 
\rblack{
\begin{tabular}{|@{\hspace{0.9mm}}c@{\hspace{1mm}}
c@{\hspace{0mm}}|c@{\hspace{1.5mm}}l@{\hspace{1.5mm}}c@{\hspace{1.5mm}}c@{\hspace{1.5mm}}r@{\hspace{0.2mm}}
c@{\hspace{2mm}}|l@{\hspace{1.5mm}}c@{\hspace{1.5mm}}c@{\hspace{1.5mm}}r@{\hspace{1.5mm}}r@{\hspace{0.9mm}}|
}
\hline
 && \multicolumn{5}{c}{S-SVRG} && \multicolumn{5}{c|}{S-SVRG-BB} \\
 \cline{3-8} \cline{9-13}
 \Gape[6pt]  retr. && $\tau^*$ & epoch & $\mathrm{\overline{nrm}}$ & $\mathrm{\overline{err}}$& $\mathrm{\overline{t}}$ &&
    epoch & $\mathrm{\overline{nrm}}$ & $\mathrm{\overline{err}}$& $\mathrm{\overline{t}}$& $\mathrm{\overline{t}ratio}$\\
\Xhline{0.6pt}
\hline \hline
&&\multicolumn{11}{@{}c|}{$r = 10$} \\ 
   exp &&  1.2 & (41,\,52.9,\,72,\,8.8) & 9e-07 & 8e-11 & 11.2&& (77,\,102.7,\,153,\,18.0) & 9e-07 & 4e-09 & 21.9 &2.0\\ 
    pd &&  1.2 & (41,\,53.0,\,73,\,8.8) & 9e-07 & 7e-11 & 8.7&& (77,\,102.8,\,153,\,18.2) & 9e-07 & 4e-09 & 16.9 &1.9\\ 
    qr &&  1.1 & (32,\,52.9,\,70,\,10.9) & 9e-07 & 1e-10 & 8.2&& (77,\,103.2,\,153,\,18.2) & 9e-07 & 4e-09 & 16.1 &2.0\\ 
    wy &&  1.2 & (41,\,52.8,\,72,\,8.8) & 9e-07 & 8e-11 & 9.8&& (77,\,102.7,\,153,\,18.0) & 9e-07 & 4e-09 & 19.0 &1.9\\ 
    jd &&  1.2 & (41,\,52.8,\,72,\,8.8) & 9e-07 & 8e-11 & 9.9&& (77,\,102.7,\,153,\,18.0) & 9e-07 & 4e-09 & 19.3 &2.0\\ 
    gp &&  1.3 & (39,\,53.4,\,73,\,10.1) & 9e-07 & 7e-11 & 7.8&& (77,\,103.5,\,153,\,17.9) & 9e-07 & 4e-09 & 15.0 &1.9\\ 
   gr &&  0.6 & (48,\,55.0,\,76,\,6.8) & 9e-07 & 8e-11 & 8.1&& (52,\,75.9,\,118,\,15.7) & 9e-07 & 2e-09 & 11.2 &1.4\\ 
\hline \hline 
&&\multicolumn{11}{@{}c|}{$r = 20$} \\ 
   exp &&  1.4 & (67,\,106.2,\,156,\,18.5) & 9e-07 & 4e-11 & 37.3&& (133,\,204.5,\,301,\,45.6) & 9e-07 & 5e-09 & 72.2 &1.9\\ 
    pd &&  1.3 & (80,\,107.2,\,128,\,12.6) & 9e-07 & 5e-11 & 23.3&& (133,\,204.3,\,300,\,45.1) & 1e-06 & 5e-09 & 44.2 &1.9\\ 
    qr &&  1.3 & (71,\,106.3,\,138,\,18.0) & 9e-07 & 5e-11 & 21.6&& (137,\,207.7,\,304,\,47.0) & 1e-06 & 5e-09 & 42.6 &2.0\\ 
    wy &&  1.4 & (66,\,106.1,\,151,\,17.9) & 9e-07 & 4e-11 & 27.9&& (133,\,204.3,\,300,\,45.4) & 1e-06 & 5e-09 & 53.5 &1.9\\ 
    jd &&  1.4 & (66,\,106.1,\,151,\,17.9) & 9e-07 & 5e-11 & 24.4&& (133,\,204.3,\,300,\,45.4) & 1e-06 & 5e-09 & 47.3 &1.9\\ 
    gp &&  1.4 & (70,\,107.3,\,139,\,17.3) & 9e-07 & 5e-11 & 20.6&& (138,\,208.8,\,307,\,47.6) & 9e-07 & 5e-09 & 39.8 &1.9\\ 
   gr &&  0.7 & (54,\,101.8,\,139,\,24.0) & 9e-07 & 5e-11 & 18.5&& (83,\,146.6,\,236,\,39.9) & 9e-07 & 3e-09 & 26.7 &1.4\\ 
\hline \hline 
&&\multicolumn{11}{@{}c|}{$r = 40$} \\ 
   exp &&  1.4 & (75,\,104.2,\,126,\,13.6) & 9e-07 & 2e-11 & 63.0&& (126,\,172.4,\,242,\,35.6) & 9e-07 & 2e-09 & 105.3 &1.7\\ 
    pd &&  1.4 & (71,\,101.8,\,125,\,16.6) & 9e-07 & 2e-11 & 34.9&& (126,\,172.4,\,242,\,35.7) & 9e-07 & 2e-09 & 58.9 &1.7\\ 
    qr &&  1.3 & (80,\,98.2,\,129,\,14.0) & 9e-07 & 3e-11 & 32.6&& (126,\,172.6,\,243,\,35.7) & 9e-07 & 2e-09 & 56.7 &1.7\\ 
    wy &&  1.4 & (74,\,103.9,\,126,\,13.9) & 9e-07 & 2e-11 & 46.5&& (126,\,172.2,\,242,\,35.5) & 1e-06 & 2e-09 & 77.7 &1.7\\ 
    jd &&  1.4 & (74,\,103.9,\,126,\,13.9) & 9e-07 & 3e-11 & 38.4&& (126,\,172.2,\,242,\,35.5) & 1e-06 & 3e-09 & 64.1 &1.7\\ 
    gp &&  1.2 & (78,\,111.2,\,138,\,18.9) & 9e-07 & 4e-11 & 35.7&& (127,\,173.8,\,268,\,37.5) & 9e-07 & 2e-09 & 55.9 &1.6\\ 
   gr &&  0.7 & (72,\,103.1,\,137,\,18.8) & 9e-07 & 3e-11 & 27.9&& (101,\,140.5,\,215,\,30.0) & 9e-07 & 1e-09 & 37.3 &1.3\\ 
\hline \hline 
&&\multicolumn{11}{@{}c|}{$r = 60$} \\ 
   exp &&  1.3 & (65,\,88.3,\,109,\,12.8) & 9e-07 & 3e-11 & 82.7&& (130,\,163.9,\,258,\,31.7) & 1e-06 & 2e-09 & 155.4 &1.9\\ 
    pd &&  1.3 & (61,\,86.3,\,109,\,14.0) & 9e-07 & 3e-11 & 40.4&& (130,\,164.2,\,255,\,31.4) & 1e-06 & 2e-09 & 76.5 &1.9\\ 
    qr &&  1.4 & (77,\,95.5,\,118,\,11.9) & 9e-07 & 2e-11 & 44.2&& (130,\,164.3,\,254,\,31.3) & 9e-07 & 2e-09 & 76.2 &1.7\\ 
    wy &&  1.3 & (65,\,87.7,\,109,\,13.2) & 9e-07 & 2e-11 & 68.7&& (130,\,164.1,\,258,\,31.5) & 1e-06 & 2e-09 & 129.0 &1.9\\ 
    jd &&  1.3 & (65,\,87.7,\,109,\,13.2) & 9e-07 & 3e-11 & 49.9&& (130,\,164.1,\,258,\,31.5) & 1e-06 & 2e-09 & 93.5 &1.9\\ 
    gp &&  1.3 & (70,\,99.0,\,130,\,17.2) & 9e-07 & 2e-11 & 41.4&& (133,\,168.2,\,272,\,33.1) & 9e-07 & 2e-09 & 70.5 &1.7\\ 
   gr &&  0.7 & (64,\,88.7,\,113,\,12.7) & 9e-07 & 2e-11 & 35.6&& (94,\,130.6,\,169,\,24.8) & 9e-07 & 8e-10 & 52.1 &1.5\\ 
\hline \hline 
\Xhline{0.6pt}
\end{tabular}
}
\end{footnotesize}
\end{table}

Finally, we plot the relative function value $(f(X^{s,0}) - f^*)/|f^*|$ for S-SVRG and S-SVRG-BB with the retraction `jd' in Figure \ref{figure:usvrg-bb-jd}, where $f^*$ is the optimal function value. From this figure, we see that both S-SVRG and S-SVRG-BB converge linearly, which is consistent with the linear convergence result shown in Theorem \ref{lemma:local:linear}. We have similar observation  for S-SVRG and S-SVRG-BB with other retractions and we omit the figures here for brevity.

\begin{figure}
\centering
\subfigure[]{\label{fig:qr_grad_kappa_1}
\centering
\includegraphics[width=2.4in]{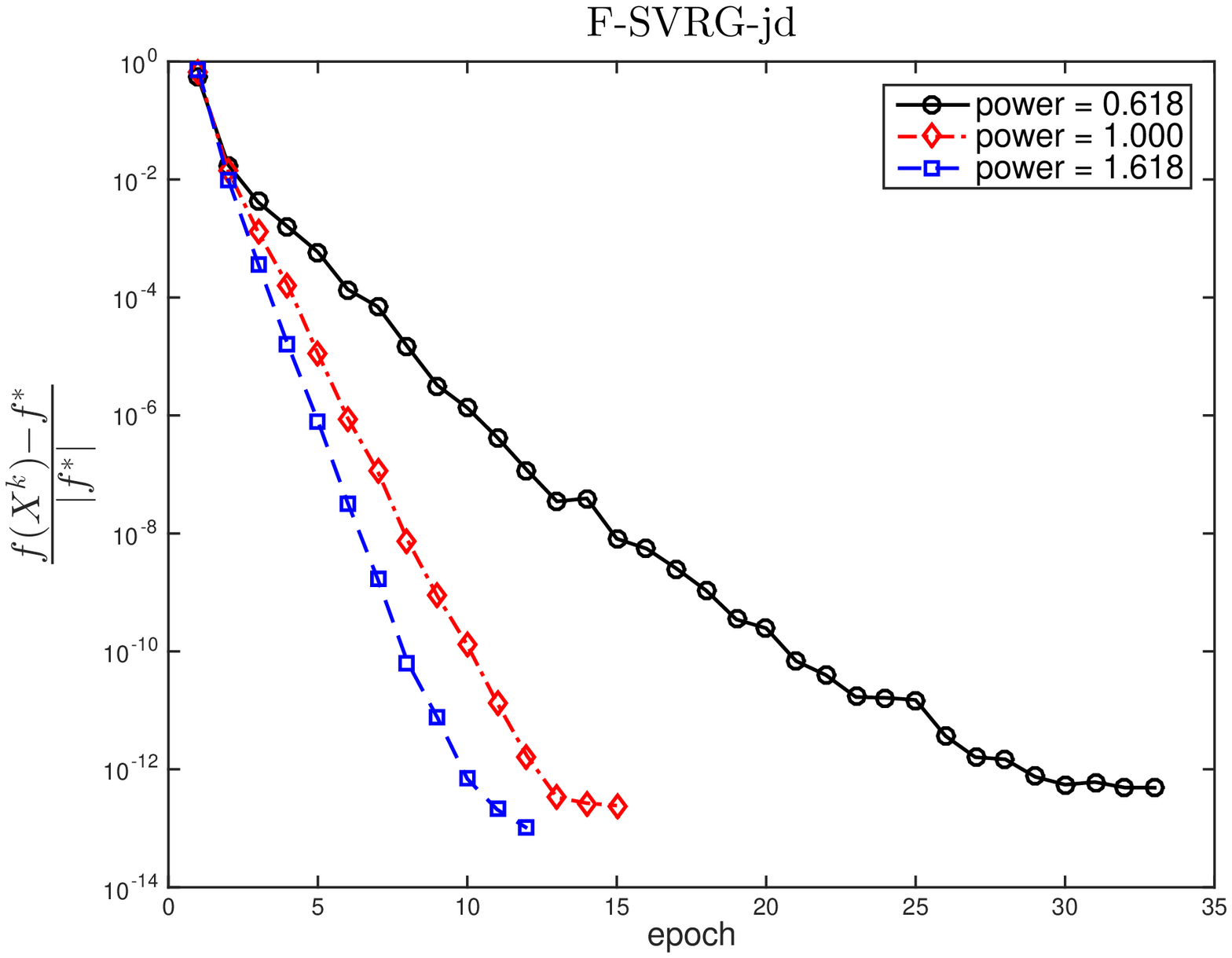}
}
\subfigure[]{\label{fig:qr_funVal_kappa_1}
\centering
\includegraphics[width=2.4in]{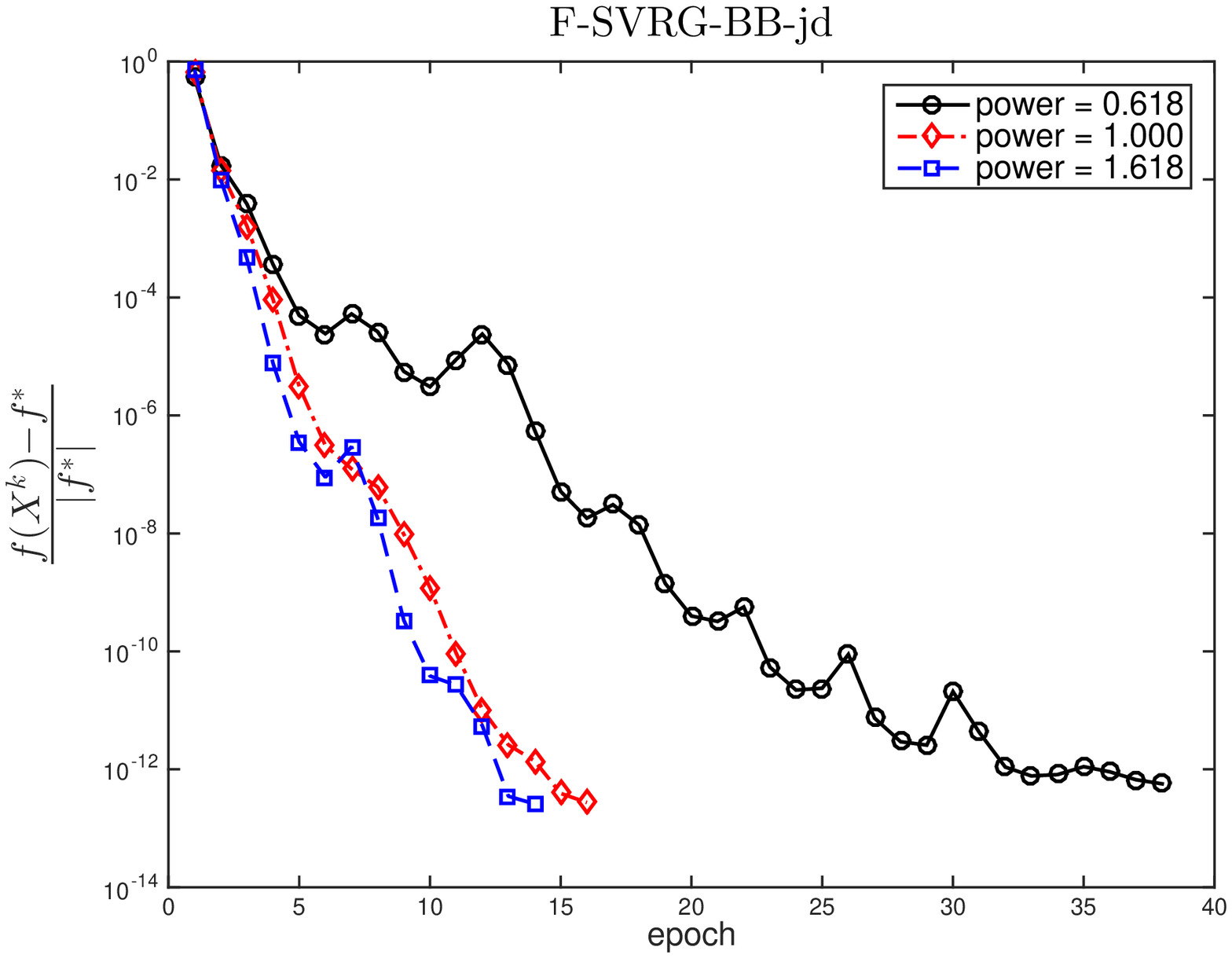}
}
\caption{The relative function value versus the number of epochs for S-SVRG and S-SVRG-BB with the retraction `jd'}
\label{figure:usvrg-bb-jd}
\end{figure}

\subsection{Matrix completion}
Let $\Omega \in \{1, \ldots, n\} \times \{1, \ldots, n\}$. For the rank-$r$ matrix $M \in \Rbb^{d \times n}$, we define the projection matrix $\proj_{\Omega}(M)$ as $\proj_{\Omega}(M)_{ij} = M_{ij}$ if $(i,j) \in \Omega$ and $\proj_{\Omega}(M)_{ij} = 0$ otherwise.
Given the observation $\proj_{\Omega}(M)$, we aim to recover missing values of $M$ by solving the following matrix completion problem \cite{kasai2016riemannian}
\be\label{prob:MC:2}
\min_{X \in \grass , a_i \in \Rbb^r} \  \frac{1}{n} \sum_{i = 1}^n \|\mathcal{P}_{\Omega_i} (X a_i) - \mathcal{P}_{\Omega_i} (M_i)\|^2,
\ee
where $M_i$ is the $i$th column of $M$, and $\proj_{\Omega_i}(\cdot)$ is the $i$th column of $\proj_{\Omega}(\cdot)$. It is easy to see that \eqref{prob:pca} is a special case of \eqref{prob:manifold} with $f_i(X) = \min_{a_i \in \Rbb^r}\  \|\mathcal{P}_{\Omega_i} (X a_i) - \mathcal{P}_{\Omega_i} (M_i)\|^2$.

We generated the synthetic data matrix $M$ as suggested in \cite{kasai2016riemannian} with the condition number set to 10. The index set $\Omega$ is chosen randomly and uniformly from $\{1, \ldots, n\} \times \{1, \ldots, n\}$, and its size is $|\Omega| = (n + d - r)r^2$.
Since the best-tuned step sizes are not easy to obtain and S-SVRG-BB has proved to be very practical, we only report the results of S-SVRG-BB. To make a fair comparison, we also adopted the BB \rblack{step size} \eqref{equ:BB} for the existing methods R-SVRG, SVRRG and VR-PCA. The corresponding methods with BB step size are named as R-SVRG-BB, SVRRG-BB and VR-PCA-BB, respectively. The method SMART-SVRG-BB stands for the SMART-SVRG with BB step size, which is essentially S-SVRG-BB with `gp' retraction. The numerical results over \rblack{20} runs are reported in Table \ref{table:mc:rsvrg:transport:new}. The term `$\mathrm{\bar t}$ratio' denotes the ratio of running time of each method over the minimal running time of R-SVRG-BB,  SVRRG-BB, VR-PCA-BB and SMART-SVRG-BB. For instance, when \rblack{$r=20$}, $d=1000$, $n=10000$, `$\mathrm{\bar t}$ratio' for S-SVRG-BB-jd is \rblack{$0.83$}, which means the CPU time of S-SVRG-BB-jd is only \rblack{$0.83$} times of that of SVRRG-BB. From Table \ref{table:mc:rsvrg:transport:new} we see that using appropriate retraction, S-SVRG-BB can be faster than the four existing methods.  It should be noted that since $\Rscr_{\mathrm{gr}}'(0) = -2 \Dbf_0(X, \nabla f(X))$ (see \eqref{equ:retraction:gr}), the BB step size for `gr' is essentially enlarged by resetting $\tau_s^{\LBB} \coloneqq 2 \cdot {\langle \mathsf{S}^s, \mathsf{S}^s\rangle}/{|\langle \mathsf{S}^s, \mathsf{Y}^s\rangle|}$. We can also enlarge the BB step size for  other retractions, and we observe that the performance is always improved.  However, for sake of space, we shall not report the corresponding results.

\begin{table}[!htbp]
\caption{Comparison of R-SVRGs' for  matrix completion instances: $d = 1000, n = 10000$}\label{table:mc:rsvrg:transport:new}
\vspace{2pt}
\centering
\begin{footnotesize} 
\begin{tabular}{|@{\hspace{0.5mm}}l@{\hspace{1mm}}|c@{\hspace{.2mm}}
l@{\hspace{1.5mm}}c@{\hspace{1.7mm}}c@{\hspace{1.7mm}}c@{\hspace{1.5mm}}c@{\hspace{0.2mm}}
c@{\hspace{1.2mm}}
|l@{\hspace{1.5mm}}c@{\hspace{1.7mm}}c@{\hspace{1.7mm}}c@{\hspace{1.5mm}}c@{\hspace{0.5mm}}|}
\hline
 \Gape[6pt]  method && epoch & $\mathrm{\overline{nrm}}$ & $\mathrm{\overline{err}}$& $\mathrm{\overline{t}}$ & $\mathrm{\overline{t}ratio}$ &&
   epoch & $\mathrm{\overline{nrm}}$ & $\mathrm{\overline{err}}$& $\mathrm{\overline{t}}$ & $\mathrm{\overline{t}ratio}$ \\
\Xhline{0.6pt} \hline \hline
&&\multicolumn{5}{@{}c}{$r = 10$} && \multicolumn{5}{@{}c|}{$r = 15$}\\ 
R-SVRG-BB && (22, \,28.6, \,75, \,11.6) & 8e-07 & 4e-11 & 222& 1.21&& (21, \,24.4, \,33, \,3.9) & 7e-07 & 5e-11 & 244 & 1.09\\ 
SVRRG-BB && (22, \,24.7, \,31, \,2.7) & 8e-07 & 5e-11 & 184& \textbf{1.00}&& (20, \,23.6, \,29, \,2.4) & 7e-07 & 6e-11 & 223 & \textbf{1.00}\\ 
VR-PCA-BB && (22, \,26.1, \,47, \,5.9) & 8e-07 & 5e-11 & 194& 1.05&& (21, \,23.8, \,31, \,2.7) & 7e-07 & 6e-11 & 227 & 1.02\\ 
SMART-SVRG-BB && (22, \,26.6, \,40, \,5.7) & 7e-07 & 5e-11 & 198& 1.07&& (21, \,23.7, \,31, \,2.7) & 7e-07 & 6e-11 & 224 & 1.00\\ 
\hdashline[2pt/2pt]
S-SVRG-BB-exp && (22, \,28.4, \,62, \,10.0) & 8e-07 & 4e-11 & 214& 1.16&& (21, \,23.9, \,32, \,3.0) & 7e-07 & 6e-11 & 229 & 1.02\\ 
S-SVRG-BB-pd && (22, \,25.2, \,40, \,4.2) & 7e-07 & 4e-11 & 188& 1.02&& (20, \,23.7, \,31, \,2.6) & 7e-07 & 6e-11 & 224 & 1.00\\ 
S-SVRG-BB-qr && (22, \,28.6, \,48, \,8.2) & 7e-07 & 4e-11 & 214& 1.16&& (20, \,24.9, \,36, \,4.8) & 7e-07 & 6e-11 & 235 & 1.05\\ 
S-SVRG-BB-wy && (22, \,25.7, \,37, \,4.5) & 7e-07 & 4e-11 & 192& 1.05&& (20, \,23.8, \,31, \,2.8) & 7e-07 & 6e-11 & 226 & 1.01\\ 
S-SVRG-BB-jd && (19, \,25.2, \,42, \,5.8) & 7e-07 & 4e-11 & 188& 1.02&& (18, \,21.4, \,29, \,2.6) & 7e-07 & 6e-11 & 203 & \textbf{0.91}\\ 
S-SVRG-BB-gr && (17, \,23.5, \,38, \,5.6) & 6e-07 & 8e-11 & 174& \textbf{0.95}&& (17, \,21.0, \,31, \,4.3) & 5e-07 & 5e-11 & 197 & \textbf{0.88}\\ 
\hline \hline 
&&\multicolumn{5}{@{}c}{$r = 20$} && \multicolumn{5}{@{}c|}{$r = 25$}\\ 
R-SVRG-BB && (22, \,24.3, \,35, \,2.9) & 7e-07 & 5e-11 & 285& 1.08&& (22, \,23.7, \,27, \,1.3) & 7e-07 & 6e-11 & 329 & 1.06\\ 
SVRRG-BB && (21, \,23.9, \,29, \,1.7) & 7e-07 & 5e-11 & 265& \textbf{1.00}&& (22, \,23.9, \,26, \,1.0) & 7e-07 & 6e-11 & 311 & 1.00\\ 
VR-PCA-BB && (22, \,24.4, \,35, \,3.3) & 7e-07 & 5e-11 & 272& 1.03&& (22, \,23.8, \,27, \,1.4) & 7e-07 & 6e-11 & 311 & \textbf{1.00}\\ 
SMART-SVRG-BB && (22, \,24.4, \,35, \,3.3) & 7e-07 & 5e-11 & 269& 1.02&& (22, \,23.8, \,27, \,1.4) & 7e-07 & 6e-11 & 307 & \textbf{0.99}\\ 
\hdashline[2pt/2pt]
S-SVRG-BB-exp && (22, \,24.8, \,38, \,4.3) & 7e-07 & 5e-11 & 277& 1.05&& (22, \,23.7, \,27, \,1.3) & 7e-07 & 6e-11 & 313 & 1.01\\ 
S-SVRG-BB-pd && (21, \,23.9, \,29, \,1.7) & 7e-07 & 5e-11 & 264& \textbf{1.00}&& (22, \,23.9, \,26, \,1.0) & 7e-07 & 6e-11 & 311 & 1.00\\ 
S-SVRG-BB-qr && (21, \,24.2, \,34, \,2.6) & 7e-07 & 5e-11 & 267& 1.01&& (22, \,23.9, \,26, \,1.0) & 7e-07 & 6e-11 & 310 & \textbf{1.00}\\ 
S-SVRG-BB-wy && (21, \,24.1, \,33, \,2.4) & 7e-07 & 5e-11 & 268& 1.01&& (22, \,23.9, \,26, \,1.0) & 7e-07 & 6e-11 & 312 & 1.00\\ 
S-SVRG-BB-jd && (18, \,19.9, \,23, \,1.7) & 7e-07 & 5e-11 & 220& \textbf{0.83}&& (18, \,19.8, \,25, \,1.5) & 6e-07 & 6e-11 & 257 & \textbf{0.83}\\ 
S-SVRG-BB-gr && (17, \,19.5, \,29, \,3.1) & 5e-07 & 5e-11 & 214& \textbf{0.81}&& (17, \,18.1, \,22, \,1.6) & 6e-07 & 7e-11 & 232 & \textbf{0.75}\\ 
\hline \hline \Xhline{0.6pt}
\end{tabular}
\end{footnotesize}
\end{table}

\section{Conclusions}\label{section:conclusion}

In this paper, we proposed a vector transport-free SVRG with general retraction for solving empirical risk minimization over manifold. 
Our S-SVRG method has several important features: (i) it can tackle general nonlinear function; (ii) it works for a variety of retractions; (iii) it formulates the unbiased and variance reduced stochastic Riemannian gradient in a simple way, without  any  additional costs such as parallel \rblack{or} vector transport. We proved that the iteration complexity of S-SVRG for obtaining a stochastic $\epsilon$-stationary point is $O(n^{2/3}/\epsilon)$, which is far less than that of Riemannian gradient descent method. With the help of \L ojasiewicz inequality, we established the linear convergence of S-SVRG. Moreover, we incorporated the BB step size to S-SVRG, and obtained a very practical S-SVRG-BB method. Numerical results on PCA and matrix completion problems showed the efficiency of the proposed methods.

\section*{Acknowledgments}
We thank Bamdev Mishra, Hiroyuki Kasai and Hiroyuki Sato for sharing their codes ``Riemannian\_svrg''.

\appendix

\section*{Appendix}

\section{Retractions on $\orth$} \label{section:retractions}
In this section, we review several retractions on $\stief$ and $\grass$.
Note that the tangent direction of the retractions of the  gradient projection and gradient reflection  are fixed, while other retractions have freedom to choose different directions. For \rblack{a} comparison of the computational cost of different retractions, see \cite{jiang2015framework, gao2016new}. 

\subsection{Retractions on $\stief$}\label{section:retraction:stiefel}
Given $X \in \stief$ and $E \in \Tbf_X \stief$, we next introduce six retractions along the direction $E$. 

\begin{enumerate}[(i)]
\item   The exponential retraction, also known as geodesic,  in \cite{edelman1998geometry}  is given as
\be
\Rscr_{\mathrm{exp1}}(X,tE) =\left[\begin{matrix}X & \mathrm{qr}(D) \end{matrix}\right] \exp \left(t \begin{bmatrix}X^{\Tsf}E &  -\mathrm{upp}(D)^{\Tsf}\\ \mathrm{upp}(D) & 0\end{bmatrix} \right)
\begin{bmatrix} I_r \\ 0\end{bmatrix},  \nn
\ee
where  $D = (I_d - XX^{\Tsf})E$ and $D = \mathrm{qr}(D) \mathrm{upp}(D)$ is the QR factorization of $D$ with $\mathrm{qr}(D) \in \orth$ and $\mathrm{upp}(D)$ being upper triangular with positive diagonal entries. 
\item The retraction of QR factorization \cite{absil2008optimization} is given as 
\be\label{equ:retraction:qr}
\Rscr_{\mathrm{qr}}(X,tE) =  \mbox{qr}(X + t E).
\ee 
\item The retraction of polar decomposition takes the form as
 \cite{manton2002optimization, absil2008optimization} 
 \be\label{equ:retraction:pd}
  \Rscr_{\mathrm{pd}}(X,tE) = \Pcal_{\orth}(X + t E),
\ee
where the projection operation $\Pcal_{\orth}(\tilde X) = \tilde U \tilde V^{\Tsf}$, where $ \tilde U \tilde \Sigma \tilde V^{\Tsf}$ is the compact SVD of $\tilde X$.  \rblack{If $\tilde X$ is full column rank, such as when $\tilde X = X + tE$, $\Pcal_{\orth}(\tilde X) = \tilde X(\tilde X^{\Tsf}\tilde X)^{-1/2}$.}

\item  Recently, based on the Cayley transformation, Wen and Yin \cite{wen2013feasible} developed a simple and efficient retraction  as \footnote{\rblack{It follows from Proposition 3.1 in \cite{jiang2015framework} that \eqref{equ:retraction:wy} is well-defined if $I_r + \frac{t}{4} X^{\Tsf} E$ is  invertible. Note that this holds naturally because  $X^{\Tsf} E$ is skew-symmetric.}}  
\be\label{equ:retraction:wy}
  \Rscr_{\mathrm{wy}}(X,tE)=X-  t U \Big(I_{2p}+\frac{t}{2}V^{\Tsf}U\Big)^{-1}V^{\Tsf}X,
\ee
where $U=[\begin{matrix}-P_X E, X\end{matrix}]$, $V=[\begin{matrix}X,\, P_XE\end{matrix}]$ with $P_X = I_d - \frac12 XX^{\Tsf}$. 
\item Later on, in the point view of subspace, Jiang and Dai \cite{jiang2015framework} proposed a family of retractions. For the generalized exponential retraction, generalized retraction of polar decomposition or QR factorization, see (8.2) - (8.4) in \cite{jiang2015framework}. Besides, \cite{jiang2015framework} also proposed a new efficient retraction as 
\be \label{equ:retraction:jd}
  \begin{cases}
\Rscr_{\mathrm{jd}}(X,tE) = (2 X + t D)J(t)^{-1} - X, \\ 
J(t) =  I_r + \frac{t^2}{4} D^{\Tsf}D - \phi(t) X^{\Tsf} E,
\end{cases} \vspace{0.2em}
\ee
 where $\phi(t)$ is any function satisfying  
 \be \label{equ:jd:phi}
 \phi(0) = 0, \quad \mbox{and}  \quad \phi'(0) = \frac12.
 \ee
   When taking $\phi(t) = \frac12 t$, \cite{jiang2015framework} showed that \eqref{equ:retraction:jd} and \eqref{equ:retraction:wy} are equivalent. 
\item Finally,  the ordinary gradient projection retraction  \cite{jiang2015framework, gao2016new} is given as 
\be \label{equ:retraction:gp}
\Rscr_{\mathrm{gp}}(X, -t \Dbf_{1/4}(X, \nabla f(X))) = \Pcal_{\orth}(X - t  \nabla f(X)).
\ee
Note that $\Rscr_{\mathrm{gp}}'(0) = -\Dbf_{1/4}(X, \nabla f(X))$ instead of \rblack{any} $E$.
\end{enumerate}

\subsection{Retractions on $\grass$}\label{section:retraction:grass}
Given $X \in \grass$ and $E \in \Tbf_X \grass$, the exponential retraction proposed in \cite{absil2008optimization} is  
\be\label{equ:retraction:exp-grass}
\Rscr_{\mathrm{exp2}}(X,tE) =  \big(X \hat V \cos \hat \Sigma t  +  \hat U \sin \hat \Sigma t\big) \hat V^{\Tsf},
\ee
where  $E = \hat U \hat \Sigma \hat V^{\Tsf}$ is the compact SVD of $E$.

Some retractions on $\stief$ can be naturally taken as the retractions on $\grass$. 
\begin{proposition}\label{proposition:grass:stief}
Suppose  $E \neq 0$, then the retractions \rblack{\eqref{equ:retraction:qr} -- \eqref{equ:retraction:jd}}
can serve as the retractions on $\grass$. 
\rblack{If $X^{\Tsf} \nabla f(X) \equiv \nabla f(X)^{\Tsf} X$, \eqref{equ:retraction:gp} is also  a retraction on $\grass$.}
\end{proposition}
\begin{proof}
It  only needs to show $\rblack{\Rscr(t)}\not \in  [X]$ for any \rblack{$t \geq 0$}. We prove this by contradiction. Suppose that $Y(t_0) \in [X]$ for \rblack{some $t_0 > 0$}, 
we have  $E^{\Tsf} Y(t_0) = 0$. 
 
 For  \eqref{equ:retraction:qr} and \eqref{equ:retraction:pd}, we have $X + t_0 E = \Rscr_{\mathrm{qr}}(t_0) \mathrm{upp}(X + t_0 E)$
and $X + t_0 E = \Rscr_{\mathrm{pd}}(t_0) (I_r + t_0^2 E^{\Tsf} E)^{\frac12}$, respectively. 
For \eqref{equ:retraction:jd}, we have $X(2I_r - J(t)) + t_0 E = \Rscr_{\mathrm{jd}}(t_0) J(t_0)$. By any of the above three equalities, we always have 
$t_0 E^{\Tsf} E = 0$, namely, $E = 0$. This leads to a contradiction.  Thus the retractions \eqref{equ:retraction:qr}, \eqref{equ:retraction:pd} and \eqref{equ:retraction:jd} are also well-defined retractions on $\grass$. 
Note that \eqref{equ:retraction:wy} is equivalent to \eqref{equ:retraction:jd} with $\phi(t) = \frac12 t$, we immediately know that \eqref{equ:retraction:wy} is also a well-defined retraction on $\grass$. 
If $X^{\Tsf} \nabla f(X) \equiv \nabla f(X)^{\Tsf} X$, the direction $E$ for \eqref{equ:retraction:gp} is given as $E = -\Dbf_0(X, \nabla f(X))$. With slight abuse of notation, let $U\Sigma V$ be the compact SVD of $X - t\nabla f(X)$. Then $Y(t_0) = UV^{\Tsf}$ and thus $E^{\Tsf} U = 0$, which further implies that $E^{\Tsf} \nabla f(X) = E^{\Tsf} E = 0$. This leads to a contradiction. 
\end{proof}

Very recently, using the Householder transformation, Gao et al. \cite{gao2016new} proposed the gradient reflection retraction as 
\be\label{equ:retraction:gr}
\Rscr_{\mathrm{gr}}(X, -2t\Dbf_0(X, \nabla f(X))) =  (-I_n + 2 \bar X(\bar X^{\Tsf} \bar X)^{\dagger}\bar X^{\Tsf})X,
\ee
where $\bar X= X - t  \nabla f(X)$ \rblack{and $(\bar X^{\Tsf} \bar X)^{\dagger}$ denotes the pseudo-inverse of $\bar X^{\Tsf} \bar X$}.
By some simple computations, we can show that $\Rscr_{\mathrm{gr}}'(0) = -2 \Dbf_0(X, \nabla f(X))$. \rblack{Similar to Proposition \ref{proposition:grass:stief}}, it is easy to show that $\Rscr_{\mathrm{gr}}(t)$ is a well-defined retraction on $\grass$. The Householder transformation is also used in \cite{sun2012low} to preserve the orthogonality constraints.

\subsection{Estimation of $L_1$ and $L_2$ for polar decomposition}\label{section:L1:L2}
\begin{lemma}\label{lemma:well:polar}
For any $X \in \orth$ and $E \in  \Tbf_X \orth$, consider the retraction of polar decomposition  \eqref{equ:retraction:pd}.
Then equations \eqref{equ:retraction:b1} and \eqref{equ:retraction:b2}  hold  for any $t \geq 0$ with $L_1 = 1$ and $L_2 = 1/2$.
\end{lemma}
\begin{proof}
First \rblack{we naturally have} $\rblack{\Rscr'(0)} = E$.  For simplicity, denote $H=  \left(I_r + t^2 E^{\Tsf} E\right)^{\frac12}$. Thus we have
$\rblack{\Rscr(t)} - \rblack{\Rscr(0)} ={}  \left( X (I_r - H) + t E \right) H^{-1}$,
which together  with the fact that $\mtr(X^{\Tsf}ES) = \mtr(S E^{\Tsf} X) = 0$ for any symmetric $S \in \Rbb^{r \times r}$  implies that
\begin{align}\label{lemma:well:polar:a1}
\|\rblack{\Rscr(t)} - \rblack{\Rscr(0)}\|_{\Fsf}^2 ={} & 
 2 \mtr (I_r - H^{-1})
={} 2 \sum_{i = 1}^r \left(1 - (1 + t^2 \lambda_i(E^{\Tsf}E))^{-1/2}\right) \rblack{\leq  t^2 \|E\|_{\Fsf}^2,}
\end{align}
where the first equality is due to $t^2 E^{\Tsf} E = H^2 - I_r$ which follows from the definition of $H$, and the inequality is due to
\rblack{$2 (1 - (1 + z)^{-1/2} ) \leq z$ with $z = t^2 \lambda_i(E^{\Tsf}E)$}.

Note that $\rblack{\Rscr(t)} - \rblack{\Rscr(0)} - t \rblack{\Rscr'(0)} =  (X+ tE) (H^{-1} - I_r)$. Again \rblack{from} $t^2 E^{\Tsf} E = H^2 - I_r$, we have
\begin{align} \label{lemma:well:polar:a2}
\|\rblack{\Rscr(t)} - \rblack{\Rscr(0)} - t \rblack{\Rscr'(0)}\|_{\Fsf}^2 ={}& 
 \mtr( (I_r - H)^2) =  \sum_{i = 1}^r \Big( 1- \sqrt{1 + t^2 \lambda_i(E^{\Tsf}E)}\,\Big)^2 \nn \\
\leq{}& \frac{t^4}{4} \sum_{i = 1}^r \lambda_i^2(E^{\Tsf}E) \leq
 \frac{t^4}{4} \|E\|_{\Fsf}^4,
\end{align}
where the first inequality is due to  \rblack{$(1 - (1 + z)^{1/2})^2 \leq z^2/4$ with $z = t^2 \lambda_i(E^{\Tsf}E)$}.
It follows from Lemma \ref{lemma:well:polar:a1} and Lemma \ref{lemma:well:polar:a2} that  \eqref{equ:retraction:b1} and \eqref{equ:retraction:b2} hold with $L_1 = 1$ and $L_2 = \frac12$, respectively.
\end{proof}
\subsection{Estimation of $L_1$ and $L_2$ for QR factorization} \label{subsection:QR}
For any $A \in \Rbb^{n \times n}$, as the same in \cite{chang1997perturbation}, we define the upper triangular matrix $\up(A) \in \Rbb^{n \times n}$ as $\up(A)_{ij} = A_{ij}$  if $i<j$, $\up(A)_{ij}  = A_{ii}/2$ if $i = j$ and $\up(A)_{ij}  = 0$ if $i > j$.
We further have that  $2\|\up(A)\|_{\Fsf}^2 =   \|A\|_{\Fsf}^2 - \frac12 \sum_{i = 1}^n A_{ii}^2 \leq \|A\|_{\Fsf}^2,$ which implies that
$\|\up(A)\|_{\Fsf} \leq  \sqrt{2}/2 \|A\|_{\Fsf}. $
\begin{lemma}\label{lemma:well:qr}
For any $X \in \orth$ and $E \in  \Tbf_X \orth$, consider the retraction of QR factorization  \eqref{equ:retraction:qr}.
Then equations \eqref{equ:retraction:b1} and \eqref{equ:retraction:b2} hold  for any $t \geq 0$ with $L_1 = 1 + \sqrt{2}/2$ and $L_2 = {\sqrt{10}}/{2}$.
\end{lemma}
\begin{proof}
Let the QR factorization of $X + tE$ be
\be\label{equ:lemma:well:qr:a1}
X + t E = Q(t) R(t),
\ee
where $Q(t) \in \orth$ and $R(t) \in \Rbb^{r \times r}$ is upper triangular with positive diagonal elements.
We then have $\rblack{\Rscr(t)} = Q(t)$ and
\be\label{equ:lemma:well:qr:RtRt}
R(t)^{\Tsf} R(t) = I_r + t^2 E^{\Tsf} E.
\ee
Differentiating   both sides of \eqref{equ:lemma:well:qr:RtRt} \rblack{with respect to} $t$, we have
$R'(t)^{\Tsf} R(t) + R(t)^{\Tsf} R'(t) = 2t E^{\Tsf}E$ and further $(R'(t) R(t)^{-1})^{\Tsf} + R'(t) R(t)^{-1} = 2t R(t)^{-\Tsf}E^{\Tsf}E R(t)^{-1}.$
Noting that  $R'(t) R(t)^{-1}$ is upper triangular,  so we obtain
\be\label{equ:lemma:well:qr:dRt}
R'(t) = 2t\,\up\left(R(t)^{-\Tsf}E^{\Tsf}E R(t)^{-1} \right) R(t).
\ee
Differentiating  both sides of \eqref{equ:lemma:well:qr:a1} \rblack{with respect to} $t$, it follows from \eqref{equ:lemma:well:qr:dRt}  that
\be \label{equ:lemma:well:qr:dQt}
Q'(t) = E R(t)^{-1} - 2t\,Q(t) \up\left(R(t)^{-\Tsf}E^{\Tsf}E R(t)^{-1} \right)\!.
\ee
We now bound the term $t \|\up\left(R(t)^{-\Tsf}E^{\Tsf}E R(t)^{-1} \right)\|_{\Fsf}$.  Using \eqref{equ:lemma:well:qr:RtRt},  it is easy to verify
\begin{align}
t^2\|R(t)^{-\Tsf}E^{\Tsf}E R(t)^{-1} \|_{\Fsf}^2 ={}& \sum_{i = 1}^r \left(\frac{t \lambda_i(E^{\Tsf} E)}{1 + t^2 \lambda_i(E^{\Tsf} E)}   \right)^2 \leq{}  \sum_{i=1}^r \min \left\{ t^2 \lambda_i^2(E^{\Tsf} E), \lambda_i(E^{\Tsf} E)/4  \right\} \nn \\
\leq{}& \|E\|_{\Fsf}^2\min\left\{t^2 \|E\|_{\Fsf}^2, 1/4\right\}, \label{equ:lemma:well:Rt_inverse_EtER_inverse:square}
\end{align}
where the first inequality uses $1 + t^2 \lambda_i(E^{\Tsf} E) \geq 2t \sqrt{\lambda_i(E^{\Tsf} E)}$. Squaring both sides of \eqref{equ:lemma:well:Rt_inverse_EtER_inverse:square} and using $\|\up(A)\|_{\Fsf} \leq  \sqrt{2}/2 \|A\|_{\Fsf}$, 
we obtain
\begin{align} \label{equ:lemma:well:qr:upper}
t \|\up\Big(R(t)^{-\Tsf}E^{\Tsf}E R(t)^{-1} \Big)\|_{\Fsf}
\leq ({\sqrt{2}}/{2}) \|E\|_{\Fsf}\min\left\{t \|E\|_{\Fsf}, 1/2\right\},
\end{align}
which \rblack{together} with \eqref{equ:lemma:well:qr:dQt} and  \eqref{equ:lemma:well:qr:dRt}, respectively, indicates
\begin{align} \label{equ:lemma:well:qr:dQt:bound}
\|Q'(t)\|_{\Fsf} \leq{} & (1 + \sqrt{2}/2)\|E\|_{\Fsf}
\end{align}
and
\begin{align} \label{equ:lemma:well:qr:dRt:bound}
\|R'(t)\|_{\Fsf} \leq \sqrt{2} \|E\|_{\Fsf}\min\left\{t \|E\|_{\Fsf}, 1/2\right\} \sqrt{1 + t^2 \|E\|_{\Fsf}^2} \leq ({\sqrt{10}}/{2}) t \|E\|_{\Fsf}^2.
\end{align}

By the Mean-Value Theorem, there exists $u \in (0,t)$ such that  $\rblack{\Rscr(t)}  - \rblack{\Rscr(0)} =  Q(t) - Q(0) = t Q'(u)$. Then
$\|\rblack{\Rscr(t)}  - \rblack{\Rscr(0)} \|_{\Fsf} = t \|Q'(u)\|_{\Fsf}$, which together with \eqref{equ:lemma:well:qr:dQt:bound} implies that \eqref{equ:retraction:b1} holds with $L_1 = 1 + \sqrt{2}/2$.
Again by the Mean-Value Theorem, noting that \eqref{equ:lemma:well:qr:a1} and $R(0) = I_r$, we have
$\rblack{\Rscr(t)}  - \rblack{\Rscr(0)} - t E=  Q(t)(R(0) - R(t)) = t Q(t)R'(u)$, where $u \in (0, t)$. Then $\|\rblack{\Rscr(t)}  - \rblack{\Rscr(0)} - t E\|_{\Fsf}  \leq t \|R'(u)\|_{\Fsf}$, which with \eqref{equ:lemma:well:qr:dRt:bound} yields  that \eqref{equ:retraction:b1} holds with $L_2 = \sqrt{10}/2$.
\end{proof}

\section{Proof of Lemma \ref{lemma:recursion}} \label{subsection:proof:Lemma:recursion}
\begin{proof}
By recursively using \eqref{equ:lemma:recursion:b:a} and noting that $\bbf_0 = 0$, we \rblack{have}
\be \label{equ:lemma:recursion:b:k_plus_1}
\bbf_{k+1} \leq  \abf \sum_{i = 0}^{k} \bbf^{k-i} \abf_i.
\ee
holds for any $k = 0, \ldots, K- 1$.
Again note that  $\bbf_0 = 0$ we thus have  from \eqref{equ:lemma:recursion:b:k_plus_1} that
\begin{align}\label{equ:lemma:recursion:bk:sum}
\sum_{k = 0}^{K-1} \bbf_k   = \sum_{k = 0}^{K-2} \bbf_{k+1} \leq{} & \abf \sum_{k = 0}^{K-2} \sum_{i = 0}^{k} \bbf^{k-i} \abf_i = \abf \sum_{k = 0}^{K-2} \left(\sum_{i=0}^{K-2-k} \bbf^i\right) \abf_{k}  =  \abf  \sum_{k = 0}^{K-1} \frac{\bbf^{K-1 - k} - 1}{\bbf - 1}  \abf_k,
\end{align}
where the last inequality is due to $\abf_k \geq 0$.  Recursively applying \eqref{equ:lemma:recursion:f:a:b} yields 
\begin{align}
\fbf_{K} \leq \fbf_0  - \cbf \sum_{k = 0}^{K-1} \abf_k + \dbf \sum_{k = 0}^{K-1} \bbf_k,
\end{align}
which together with \eqref{equ:lemma:recursion:bk:sum} yields  \eqref{equ:lemma:recursion:bk:fK}.
\end{proof}

\section{Proof of Theorem \ref{theorem:L:inequality:MC}}\label{section:L:Grassmann}

First,  define $d_c(\ubf, \ubf^*) = (d_c(W,U)^2 + d_c(Z,V)^2)^{\frac12}$ with
\be\label{equ:def:dcXU}
d_c(W,U) =  \frac{1}{\sqrt{n}} \min_{Q_1, Q_2 \in \mathsf{St}_r} \|W Q_1 - U Q_2\|_{\Fsf} = \rblack{\frac{1}{\sqrt{n}} \min_{Q \in \mathsf{St}_r} \|W Q - U \|_{\Fsf}}.
\ee
It is known that $d_c(\ubf, \ubf^*) \leq d(\ubf, \ubf^*)$ (see, e.g., Remark 6.1 in \cite{keshavan2010matrix}).  We now present a useful proposition.

\begin{proposition}
Suppose that $\ubf^* \in \mathsf{M}(m,n) \cap \Kcal(3\mu_0)$.  Then
\begin{align}\label{equ:tF:upperbound}
\tF(\ubf) \equiv \tF(W,Z) \leq n \left(m \Sigma_{\max}^2 + \frac{14 e^{\frac19}}{9 \mu_0 r}  \varrho \right) d(\ubf, \ubf^*)^2
\end{align}
holds for all $\ubf \in \mathsf{M}(m,n) \cap \Kcal(4\mu_0)$ with $\mathsf{M}(m,n) = \mathsf{g}(m,r) \times \mathsf{g}(n,r)$.
\end{proposition}
\begin{proof}
First,  we have
\begin{align} \label{equ:FXY}
F(W,Z) ={} & \frac12 \|\proj_{\Omega} (M - WSZ^{\Tsf})\|_{\Fsf}^2 \leq{} \frac12 \|\proj_{\Omega} (M - W\Sigma Z^{\Tsf})\|_{\Fsf}^2 \nn \\
\leq{} & \frac12 \|M - W\Sigma Z^{\Tsf}\|_{\Fsf}^2  = \frac12 \|U \Sigma V^{\Tsf}  - W \Sigma V^{\Tsf}  + W \Sigma V^{\Tsf}  - W\Sigma Z^{\Tsf}\|_{\Fsf}^2 \nn \\
\leq{}& \| (U - W) \Sigma V^{\Tsf} \|_{\Fsf}^2 + \| W \Sigma (V - Z)^{\Tsf} \|_{\Fsf}^2  \nn \\
\leq{}& m  \Sigma_{\max}^2 \left(\|U - W\|_{\Fsf}^2 + \|V - Z\|_{\Fsf}^2 \right)\!,
\end{align}
where the first inequality is due to the optimality of $S$.

Now, let us bound the last two terms in \eqref{equ:def:tF}. It is easy to show that
\be \label{equ:G1z}
G_1(z) \leq e^{\frac19} (z - 1)^2, \quad \forall z \in [0, 4/3].
\ee
Note that $\ubf \in \mathsf{M}(m,n) \cap \Kcal(4\mu_0)$ implies that $\frac{\|W^{(i)}\|^2}{3 \mu_0 r} \leq \frac{4}{3}$. Define
$$\mathcal{I}_1 = \left\{i : \frac{\|W^{(i)}\|^2}{3 \mu_0 r} \leq 1, i\in\{1, \dots, m\}\right\}, \  \mathcal{I}_2 = \left\{i : 1 < \frac{\|W^{(i)}\|^2}{3 \mu_0 r} \leq \frac43, i\in \{1, \dots, m\}\right\}.$$
It follows from \eqref{equ:G1} and  \eqref{equ:G1z}
that
\begin{align}
\sum_{i = 1}^m G_1 \left( \frac{\|W^{(i)}\|^2}{3 \mu_0 r}\right)  ={}&\sum_{i \in \mathcal{I}_2} G_1 \left( \frac{\|W^{(i)}\|^2}{3 \mu_0 r}\right) \nn \\
 \leq{} & e^{\frac19} \sum_{i \in \mathcal{I}_2} \left( \frac{\|W^{(i)}\|^2}{3 \mu_0 r} - 1 \right)^2
\leq e^{\frac19} \sum_{i \in \mathcal{I}_2} \left( \frac{\|W^{(i)}\|^2}{3 \mu_0 r} -  \frac{\|U^{(i)}\|^2}{3 \mu_0 r} \right)^2  \nn \\
={}&  \frac{e^{\frac19}}{9\mu_0^2 r^2} \sum_{i \in \mathcal{I}_2} \left(\|W^{(i)}\| - \|U^{(i)}\|\right)^2  \left(\|W^{(i)}\| + \|U^{(i)}\|\right)^2 \nn \\
\leq{}&  \frac{2 e^{\frac19}}{9 \mu_0^2 r^2} \sum_{i \in \mathcal{I}_2} \left(\|W^{(i)} - U^{(i)}\|^2\right) \left(\|W^{(i)}\|^2 + \|U^{(i)}\|^2 \right) \nn \\
\leq{}& \frac{14 e^{\frac19}}{9 \mu_0 r}\sum_{i \in \mathcal{I}_2} \left(\|W^{(i)} - U^{(i)}\|^2\right) \leq  \frac{14 e^{\frac19}}{9 \mu_0 r} \|W - U \|_{\Fsf}^2,
\label{equ:G1Xbound}
\end{align}
where the second inequality is due to  $\frac{\|U^{(i)}\|^2}{3 \mu_0 r} \leq 1$, and the fourth inequality uses the facts that $\frac{\|U^{(i)}\|^2}{3 \mu_0 r} \leq 1$ and $\frac{\|W^{(i)}\|^2}{3 \mu_0 r} \leq \frac{4}{3}$.
Similarly, we have
\be
\sum_{j = 1}^n G_1 \left( \frac{\|Z^{(j)}\|^2}{3 \mu_0 r}\right)  \leq  \frac{14 e^{\frac19}}{9 \mu_0 r} \|Z -  V\|_{\Fsf}^2. \label{equ:G1Ybound}
\ee
Combining \eqref{equ:FXY}, \eqref{equ:G1Xbound} and \eqref{equ:G1Ybound}, we have
\be
\tF(\ubf) \equiv \tF(W,Z) \leq  \left(m  \Sigma_{\max}^2 + \frac{14 e^{\frac19}}{9 \mu_0 r}  \varrho \right) \left( \|U - W\|_{\Fsf}^2 + \|V - Z\|_{\Fsf}^2 \right)
\nn
\ee
for any $\ubf \in \mathsf{M}(m,n) \cap \Kcal(4\mu_0)$. Note that $\tF(W,Z) \equiv \rblack{\tF(WQ_{W},ZQ_{Z})}$ for any \rblack{$Q_W, Q_Z \in \mathsf{St}_r$}.  By the definition \eqref{equ:def:dcXU} of $d_c(\ubf, \ubf^*)$ and $d_c(\ubf, \ubf^*) \leq d(\ubf, \ubf^*)$, we arrive at \eqref{equ:tF:upperbound}.
\end{proof}

 Second, it follows from Lemma 6.5 in \cite{keshavan2010matrix} that
\be \label{equ:gradtF:lowerbound}
\|\rgrad \tF(\ubf)\|^2 \geq C n \epsilon^2 \Sigma_{\min}^4 d(\ubf, \ubf^*)^2
\ee
for all $\ubf \in \mathsf{M}(m,n) \cap \Kcal(4\mu_0)$ and $d(\ubf, \ubf^*) \leq \delta$ with probability at least
$1 - 1/n^4$.

Finally, combing \eqref{equ:gradtF:lowerbound} and  \eqref{equ:tF:upperbound}, noting that $\tF(\ubf^*) = 0$,  we have Theorem \ref{theorem:L:inequality:MC}. 
\section{Proofs for  Theorem \ref{theorem:complexity:bounded:g} and  Theorem \ref{theorem:complexity:no:bounded:g}}
\subsection{Proof for  Theorem \ref{theorem:complexity:bounded:g}}\label{section:proof:bounded:g}
For fixed $s$, \rblack{we again drop the subscript $s$} for simplicity. Similar to Lemma \ref{lemma:sufficient:decrease:epoch}, we  have
\begin{align}\label{equ:f:x:K:expectation:bounded:g}
\Ebb_{\xi_{[K-1]}}[\mF(X^{K})]  \leq f(X^0) -  \sum_{k=0}^{K-1}\Delta_k \Ebb_{\xi_{[K-1]}}\big[\| \rgrad\,f(X^k) \|_{\Fsf}^2 \big], 
\end{align}
where \rblack{$\Delta_k$ is given in  \eqref{equ:Delta:bounded:g}.}
 The proof of \eqref{equ:f:x:K:expectation:bounded:g} is the same as that of Lemma \ref{lemma:sufficient:decrease:epoch}  except that  $\|X^k - X^0\|_{\Fsf} \leq 2\sqrt{r}$ is replaced by
\rblack{$\|X^k - X^0\|_{\Fsf} \leq 3C^{\Mcal}L_1^{\Mcal} K$,}
because
\begin{align}
\|X^k - X^0\|_{\Fsf} \leq{} & \sum_{j =1}^{k}  \|X^j - X^{j-1} \|_{\Fsf} =   \sum_{j =1}^{k}  \| \Rscr\big({X^{j-1}}, -\tau \Dcal(X^{j-1}, \xi_{j-1})\big) - \Rscr(X^{j-1},0) \|_{\Fsf} \nn \\
\leq{}& \sum_{j = 1}^k \rblack{L_1^{\Mcal}} \tau \|\Dcal(X^{j-1}, \xi_{j-1})\|_{\Fsf} \leq \sum_{j = 1}^k \rblack{L_1^{\Mcal}}  \tau \|\Gcal(X^{j-1}, \xi_{j-1})\|_{\Fsf} \leq 3C^{\Mcal} \rblack{L_1^{\Mcal}} K \tau, \nn
\end{align}
where the second inequality is due to \rblack{\eqref{equ:retraction:b1:general}}.

Note that Theorem \ref{theorem:svrg:grad:norm} still holds.  Next we estimate $\Delta_{\min}$.
Again we have $\Delta_{\min}  = \Delta_0$.
\rblack{Note that $\tilde{L}^{\Mcal} = \tilde{L}^{\Mcal}_1 + \tilde{L}^{\Mcal}_2$ and $\tilde{L}^{\Mcal}_1 \geq 1$, together with \eqref{equ:K:svrg:bounded:g} and $0 \leq \mu \leq 2/3$,}  we obtain
\be\label{equ:section:proof:bounded:g:00:11}
\frac{\tilde{L}^{\Mcal} L^2\tau^2}{\nu^2 |\Bsf|}   \leq  c^2K^{\mu - 2}, \quad 
\rblack{\frac{\tilde{L}_1^{\Mcal} + \tilde{L}_2^{\Mcal} K \tau}{\tilde{L}^{\Mcal}}  + \frac{2}{\tilde{L}^{\Mcal} \beta \tau}\leq K^{1- \mu} + \frac{2K}{c}}.
\ee
With the first assertion in \eqref{equ:section:proof:bounded:g:00:11} and \eqref{equ:K:svrg:bounded:g}, we have
\be\label{equ:usvrg:tau:beta:001:general}
\rblack{\Gamma_0^{\Mcal}  \leq  \frac{\exp(c^2 + 2c) - 1}{c^2 + 2c} K.}
\ee
Using \eqref{equ:section:proof:bounded:g:00:11} and \eqref{equ:usvrg:tau:beta:001:general},  by  \rblack{direct} calculations,  we obtain from \eqref{equ:section:proof:bounded:g:gamma:k} \rblack{with $k = 0$} and \eqref{equ:K:svrg:bounded:g} that
%
$$
\frac{\Delta_0}{\tau} \geq  \nu - \frac{\nu}{2} \frac{\hat L^{\Mcal}}{\sqrt{\tilde{L}^{\Mcal}} L} \exp(c^2 + 2c)  c  \geq \frac12 \nu,
$$
where the second inequality is due to \eqref{equ:c:bounded:g}.

Similar to  the proof for Theorem \ref{theorem:complexity}, we arrive at  Theorem \ref{theorem:complexity:bounded:g}.

\subsection{Proof for   Theorem \ref{theorem:complexity:no:bounded:g}}\label{section:proof:no:bounded:g}
For fixed $s$,  we again drop the subscript $s$ for simplicity. Similar to Lemma \ref{lemma:sufficient:decrease:epoch}, we  have
\begin{align} \label{equ:f:x:K:expectation:no:bounded:g}
\Ebb_{\xi_{[K-1]}}[\mF(X^{K})]  \leq f(X^0) -  \sum_{k=0}^{K-1}\Delta_k \Ebb_{\xi_{[K-1]}}\big[\| \rgrad\,f(X^k) \|_{\Fsf}^2 \big],
\end{align}
where \rblack{$\Delta_k$ is given in \eqref{equ:Delta:no:bounded:g}.}
The proof of \eqref{equ:f:x:K:expectation:no:bounded:g} is the same as that of Lemma \ref{lemma:sufficient:decrease:epoch} except that \eqref{equ:lemma:Xk:update:a1} is replaced by
\begin{align}
\|X^{k+1} - X^0\|_{\Fsf}^2  
\leq{} & \left(1 + \beta \right) \|X^k - X^0\|_{\Fsf}^2  + \Big(1 + \frac{1}{\beta} \Big)  \|\Rscr\big(X^k, -\tau \Dcal (X^k,\xi_k)\big) - \Rscr\big(X^k, 0\big)\|_{\Fsf}^2  \nn \\
\leq{} &  \left(1 + \beta \right) \|X^k - X^0\|_{\Fsf}^2  +   (L_1^{\Mcal} \tau)^2 \Big(1  + \frac{1}{\beta}\Big)\|\Dcal (X^k,\xi_k)\|_{\Fsf}^2,
\label{equ:lemma:Xk:update:no:bounded:g:a1}
\end{align}
where the first inequality is due to the Cauchy-Schwarz inequality, $X^k = \Rscr(X^k, 0)$ and  \eqref{equ:Xk:update:usvrg}, and the second inequality is  due to \rblack{\eqref{equ:retraction:b1:general}}.


Note that Theorem \ref{theorem:svrg:grad:norm} still holds.  Next we estimate $\Delta_{\min}$.
Again we have $\Delta_{\min}  = \Delta_0$.
We can obtain from \eqref{equ:K:svrg:no:bounded:g} that
\be \label{equ:section:proof:no:bounded:g:00:11}
\rblack{\frac{L^2 (L_1^{\Mcal})^2 \tau^2 }{\nu^2 |\Bsf|}   \leq \frac{c^2}{K^2},\ \Gamma_0^k \leq \frac{\exp(c^2 + 2c) - 1}{c + 2} K}.
\ee
\rblack{Noting that $\tau \leq  c\nu/(L L_1^{\Mcal})$},  \rblack{by some simple calculations, we know from}  \eqref{equ:section:proof:no:bounded:g:00:11}, \eqref{equ:K:svrg:no:bounded:g}, \eqref{equ:c:no:bounded:g} \rblack{and \eqref{equ:Delta:no:bounded:g} with $k=0$}  that
\begin{align} \label{equ:usvrg:tau:beta:Delta0:divide:tau:final:no:bounded:g}
 \frac{\Delta_0}{\tau} \geq{}  \nu - \frac{\nu}{2} \frac{\hat L^{\Mcal}}{L L_1^{\Mcal}} \left(\frac{c+1}{c+2} \exp(c^2 + 2c) + \frac{1}{c+2}\right) c \geq \frac{\nu}{2},
\end{align}
and thus $\Delta_{\min} \geq \nu\tau/2$.

Similar to  the proof for Theorem \ref{theorem:complexity}, we arrive at  Theorem \ref{theorem:complexity:no:bounded:g}.

\bibliography{svrgOptM17}
\bibliographystyle{siam}

\end{document}